\documentclass{amsart}
\usepackage{amsthm}
\usepackage{amsmath}
\usepackage{amsrefs}
\usepackage{amsfonts}
\usepackage{verbatim}
\usepackage{hyperref}
\usepackage{enumitem}

\theoremstyle{plain}
\newtheorem{thm}{Theorem}[section]

\theoremstyle{plain}
\newtheorem{lem}[thm]{Lemma}
\theoremstyle{definition}
\newtheorem{mydef}[thm]{Definition}

\theoremstyle{remark}
\newtheorem{rem}[thm]{Remark} 

\DeclareMathOperator{\supp}{supp}

\DeclareMathOperator{\Imag}{Im}
\renewcommand{\Im}{\Imag}
\renewcommand{\Re}{\text{Re\,}}

\newcommand{\N}{\mathbb{N}} 
\newcommand{\Z}{\mathbb{Z}} 
\newcommand{\R}{\mathbb{R}} 
\newcommand{\C}{\mathbb{C}} 
\newcommand{\grad}{\nabla}
\newcommand{\del}{\partial}
\newcommand{\XuDu}{\ensuremath{x, t, u, \bar{u}, \grad u, \grad \bar u}}

\newcommand{\Xu}{\ensuremath{x, t, u, \bar{u}}}
\newcommand{\Xv}{\ensuremath{x, t, v, \bar{v}}}
\newcommand{\abs}[1]{\ensuremath{\left|#1\right|}}
\newcommand{\ip}[2]{\ensuremath{\left\langle #1, #2 \right\rangle}}
\newcommand{\jap}[1]{\ensuremath{\langle#1\rangle}}

\newcommand{\lp}{\ensuremath{\left(}}
\newcommand{\rp}{\ensuremath{\right)}}%
\newcommand{\paren}[1]{\ensuremath{\lp #1 \rp}}
\newcommand{\norm}[1]{\ensuremath{\left\| #1 \right\|}}
\newcommand{\tnorm}[1]{\ensuremath{||| #1 |||}}
\newcommand{\linref}[1]{\ref{#1}}
\newcommand{\nlinref}[1]{\ref{#1}}

\newcommand{\diff}[2]{\ensuremath{\frac{\del #1}{\del #2}}}

\begin{document}
\title{Quasilinear Schr\"odinger Equations}
\author{Nicholas P. Michalowski}

\maketitle

\begin{abstract}
  In this paper we prove local well-posedness for Quasi-linear
  Scrh\"odinger equations with initial data in unweighted Sobolev
  Spaces. For small initial data with minimal smoothness this has
  addressed by J.~Marzuola, J.~Metcalfe and D.~Tataru \cite{JMJMDT2012},
  \cite{JMJMDT2012-2}.  This work does not attempt to address the
  minimal regularity for initial data, but instead builds on the
  previous results of C.~Kenig, G.~Ponce, and L.~Vega
  \cite{CKGPLV1998}, \cite{CKGPLV2004}
  to remove the smallness condition in unweighted spaces. 
  This is accomplished by developing a non-centered version of Doi's
  Lemma, which allows one to prove Kato type smoothing estimates.
  These estimates make it possible to achieve the necessary a priori
  linear results.  
\end{abstract}

\section{Introduction} We are interested in the local solvability
of the IVP
\begin{equation}\label{theproblem}
\left\{\begin{aligned}
&\del_t u = ia_{jk}(\XuDu)\del_{x_j}\del_{x_k}u +
\vec b_1(\XuDu)\cdot\grad u\\
&\qquad + \vec{b}_2(\XuDu)\cdot\grad \bar u +c_1(\Xu)u +
c_2(\Xu)\bar u +f(x,t)\\
&u(x,0)=u_0(x).
\end{aligned}\right.
\end{equation}

Quasi-linear Schr\"odinger equations have been studied extensively in
recent years.  The aim of the current work is to extend some results of
C.~Kenig, G.~Ponce, and L.~Vega in \cite{CKGPLV2004}.  In
particular we aimed to remove the assumption that the initial data
$\jap{x}^2\del_x^\alpha u_0\in L^2$ for suitable $\alpha.$

As pointed out in \cite{CKGPLV2004}, other forms of these equations
have been extensively studied.  In \cite{CKGPLV1993}, the same authors
show that the equation
\begin{equation}\label{constcoeffNLS}
\del_t u +i\scr{L}u=P(u,\bar u, \grad u, \grad \bar u)\end{equation}
 with \[\scr{L}=\sum_{i=1}^k\del_{x_i}-\sum_{i=k-1}^n\del_{x_k}\]
and $P(\cdot)$ a non-linearity, is locally well posed for small
initial data in $H^s.$  The smallness condition was first removed in
$n=1$ by N.~Hayashi and T.~Ozawa in \cite{NHTO1994}.  After a change
of variables they were able to write the equation as an equivalent
system that did not involve first order terms in $u$.  For this system
can be handled by the energy method.

For the case elliptic case when $\scr{L}=\Delta$, H.~Chihara
\cite{HC1995} was able to remove the smallness condition in all
dimensions.  Again, the main idea here was to use a transformation
which eliminates the first order terms in $u$ so that the energy
method applies.  For the change of variables to cancel the first
order terms it was necessary to first diagonalize the system for
$(u,\bar u).$  In order to diagonalize the system, as we will see
below, the ellipticity of $\scr{L}$ is essential.

In \cite{CKGPLV1998}, Kenig et.\ al.\ removed the smallness condition
in all dimensions.  They construct a pseudo-differential operator $C$
so that $\overline{Cv}=C\overline{v}$, and because of this they are
able to avoid the diagonalization argument needed in \cite{HC1995}.
The construction of $C$ produces a symbol in the
Calder\'on-Vaillancourt class.

As one moves to variable coefficient second order terms it becomes
necessary to introduce non-trapping conditions on the coefficients.
Consider the equation
\begin{equation}
  \left\{
    \begin{aligned}
      &\del_tu=i\del_{x_k}a_{kj}(x)\del_{x_j}u +
          \vec b_1(x)\cdot \grad u +f\\
      &u|_{t=0}=u_0
    \end{aligned}
\right.
\end{equation}

where $a_{kj}$ elliptic and asymptotically flat. Ichinoise
\cite{WI1984} show that
\[\sup_{\stackrel{\stackrel{x_0\in\R^n}{\xi_0\in S^{n-1}}}{t_0\in\R}}
\abs{\int_0^{t_0}\Im \vec b_1(X(s,x_0,\xi_0))\cdot\Xi(s,x_0,\xi_0)\,ds}\]
 is a necessary condition for the
 estimate \[\sup_{0<t<T}\norm{u}_{L^2}\leq C_T\paren{\norm{u_0}_{L^2}+
\norm{f}_{L^1_T L^2_x}}.\]

The non-trapping assumption  is closely related to local smoothing
estimates, which are key to the linear theory.  This can be seen from
the work of S. Doi (\cite{SD1994},\cite{SD1996}), Craig et.\
al.\ \cite{WCTKWS1995}, and others.  From their work it can be seen
that, under appropriate smoothness, ellipticity and asymptotic
flatness assumptions, the non-trapping condition for
\begin{equation}
  \left\{
    \begin{aligned}
      &\del_tu=i\del_{x_k}a_{kj}(x)\del_{x_j}u \\
      & u|_{t=0}=u_0
    \end{aligned}\right.
\end{equation}
verify local smoothing estimates.  That is, estimates of the form\\
$\norm{J^{1/2}u}_{L^2(\R^n\times[0,T],\jap{x}^m\,dxdt)}\leq
C_T\norm{u_0}_{L^2}$.
In addition, Doi \cite{SD2000} also showed that, under the same
conditions, if the above estimate holds then the non-trapping
assumption must hold.

C.\ Kenig et al in (\cite{CKGPCRLV2005}, \cite{CKGPCRLV2006}) have
extended the results of their previous work in the variable
coefficient case by removing ellipticity assumptions.  Their work
assumes that the initial data is in a weighted Sobolev space.  It will
be the subject of future work to extend the methods here to remove the
weights in this cases.

Recently, in both the elliptic and hyberbolic settings J.~Marzuola,
J.~Metcalfe and D.~Tataru \cite{JMJMDT2012-2} have established low
regularity local well-posedness for for small initial data in $H^s$
for $s>(n+5)/2$.  Having a smallness condition on the initial data
allows the authors to avoid explicit non-trapping assumptions.  In
\cite{JMJMDT2012} the above authors also considered the
situation in which quadratic interactions are present and establish
low regularity well-posedness results. 

Our own contribution to this body of research is remove the smallness
condition for the work of Kenig, Ponce, Vega without imposing any
smallness on conditions on the initial data.  

Specifically, we assume the following conditions on the coefficients.  Let $B_M^k(0)=\{z\in \C^k\mid \abs{z}<M\}.$
\begin{enumerate}[labelsep=\parindent, leftmargin=*, label=(NL\arabic*)]
\item\label{NLRegularity} There exist $\tilde N=\tilde N(n)\in\N$ such
  that, for any $M>0$, $a_{jk}, b_{1,j}, b_{2,j}\in
  C_b^{\tilde N}(\R^{n}\times\R\times B_M^{2n+2}(0))$ for $j, k=1, 2, \ldots,
  n$, and $c_1, c_2\in C_b^{\tilde N}(\R^n\times \R \times B^2_M(0)).$
\item\label{NLReal2ndOrder} Let $(x, t, \vec z)\in \R^n\times \R\times
  \C^{2n+2}.$  The matrix  $A(x, t, \vec z) =
  \paren{a_{jk}(x, t, \vec z)}_{j,k=1, \ldots, n}$ is real valued.
\item\label{NLSymmetric} The matrix $A(x,t,\vec z)$ is symmetric for
  all $x\in \R^n$, $t\in \R$ and $\vec z\in B_M^{2n+2}(0).$

\item\label{NLElliptic} For $\vec z \in B_M^{2n+2}(0)$ the matrix
 $A(x,t,\vec z)$ is uniformly elliptic.  That is, there exists
 $\gamma_M>0$ such that
\[\gamma_M\abs{\xi}^2\leq \sum_{j,k=1}^n a_{jk}(x,t,\vec
z)\xi_j\xi_k\leq \gamma_M^{-1}\abs{\xi}^2,\]
for all $x\in \R^n$, $t\in \R$ and $\vec z\in B_M^{2n+2}(0).$

\item\label{NLFlat} The matrix $A(x,t,\vec z)$ is asymptotically flat.
  That is, there exists a constant $C_M$ such that $I-A(x,t,\vec z)$,
  together with any derivatives of $A(x,t,\vec z)$ up to order 2
  (not including $\del_t^2 A(x,t,\vec z)$), are bounded in absolute value
  by $\frac{C_M}{\jap{x}^2}$.

\item\label{NLFirstOrder}
  Here and throughout we let
  $\R^n=\bigcup_{\mu\in\Z^n} Q_\mu$ where $Q_\mu$ are unit cubes with
  vertices in $\Z^n$ and centers $x_\mu$.
  Suppose that, for $j=1,2$,  $b_j(x,t,0,0,\vec 0,\vec 0)=0$, and
  $\del_{z_i} b_j(x,t,0,0,\vec 0,\vec 0)=0$.

  Also, for some $C_M>0$, we have that $\del_{x_l}a_{jk}(x,t,\vec z) =
  \sum_{\mu\in\Z^n}\alpha_\mu\phi_\mu^{ljk}(x,t,\vec z)$ with
  $\alpha_\mu>0$, $\sum \alpha_\mu<C_M$, $\phi_\mu^{ljk}(\cdot,t,\vec
  z)\in C^{\tilde N}(\R^n)$ with $\norm{\phi_\mu^{ljk}}_{C^{\tilde N}}\leq 1$
  and uniformly for $t\in \R$ and $\vec z\in B_M^{2n+2}(0)$ we have
  $\supp \phi_{\mu}^{ljk}\subseteq Q_\mu^*$ (the double of $Q_\mu$)
  for $l,k,j=1\ldots n$.  Similarly for $\del_ta_{jk}$,
  $\del_{z_m}a_{jk}$, and $\del_t\del_{z_m}a_{jk}$.

\item\label{NLNoTrap} We associate to our coefficients and our initial
  data the symbol $$h(x,\xi)=-a_{jk}(x,0,u_0,\bar u_0,\grad
  u_0,\grad\bar u_0)\xi_j\xi_k .$$  We assume the bicharacteristic flow
  obtained from $h$ is non-trapping.  That is the solution to the
  system of ODE's
 \[\left\{
     \begin{aligned}
       &\frac{d}{dt}X_j(s,x,\xi)=\diff{h}{\xi_j}(X,\Xi) \\
       &\frac{d}{dt}\Xi_j(s,x,\xi)=-\diff{h}{x_j}(X,\Xi) \\
       & X(0,x,\xi)=x \quad\text{and}\quad\Xi(0,x,\xi)=\xi
     \end{aligned}
\right.\]

satisfies $\{X(s,x,\xi)\mid s>0\}$ and $\{X(s,x,\xi) \mid s<0\}$ are
unbounded for all $(x,\xi)\in\R^n\times\paren{\R^n\setminus\{0\}}$.
\end{enumerate}

\begin{thm}\label{solution}
  Under these assumptions there exist $\tilde N, s$ depending only on the
  dimension so that if we are given $u_0\in H^s$ and $f\in
  L^\infty([0,1];H^s)$, then there is a $T_0<1$
  depending on the norms of $u_0$ and $f$ and
  \nlinref{NLRegularity}-\nlinref{NLNoTrap} so that there is a unique
  solution to \eqref{theproblem}, $u(x,t)$, on the interval $[0,T_0]$
  satisfying
  \[u\in C([0,T_0];H^{s-1})\cap L^\infty([0,T_0];H^s).\]
\end{thm}

The remainder of the paper is organized as follows.  In
section~\ref{chap:doi} we establish an uncentered version of Doi's
lemma necessary to later results.  In section~\ref{sec:linear} we
establish a priori linear results.  Finally, in
section~\ref{sec:nonlinear}, we give the proof of Theorem~\ref{solution}

\section{Doi's Lemma}\label{chap:doi}

Doi's lemma is a key estimate that allows us to obtain local
smoothing.  It is the local smoothing estimates that allow us to
handle the first order terms in the linear theory.  In this section we
present two variants of Doi's lemma, one due to S.~Doi that holds in
the elliptic setting and one due to C.~Kenig, G.~Ponce, C.~Rolvung,
and L.~Vega in \cite{CKGPCRLV2005}, that also holds when the
coefficients are not necessarily elliptic. We then show how to extend
these results to corresponding ``non-centered'' versions that we need
for the precise form of our local smoothing estimates.

We consider the symbol $h(x,\xi)\in S^2_{1,0}$ defined by
$h(x,\xi)=\sum_{j,k=1}^na_{jk}(x)\xi_j\xi_k$.  Let $A(x)$ denote the
matrix $\left(a_{jk}(x)\right)_{j,k=1}^n$.  We impose the following
assumptions on $A(x)$

\begin{enumerate}[labelindent=0in, leftmargin=*, label=(D\arabic*)]
\item\label{D1} There exist $N=N(n)\in\N$, and $C>0$  so that  $a_{jk}(x)\in
  C_b^N(\R^n)$ for $j, k = 1, 2,\ldots, n,$ with norm controlled by $C$.
\item\label{D2} The functions  $a_{jk}(x)$ are real valued and the matrix $A(x)=\big(a_{jk}(x)\big)_{jk=1,...,n}$ is symmetric.
\item\label{D3} The matrix $A(x)$ is uniformly elliptic.  That is, for
  all $x\in \R^n$ there is a positive number $\gamma$ so that
$$C^{-1} |\xi|^2 \leq \sum_{j, k = 1}^n a_{jk}(x)\xi_j\xi_k \leq
C|\xi|^2.$$
\item\label{D4} The matrix $A(x)$ is asymptotically flat. That is
$$\abs{I-A(x)} \leq \frac{C}{\jap{x}^2} \qquad
\text{and}\qquad \abs{\grad_x a_{jk}(x)} \leq \frac{C}{\jap{x}^2}.$$ 
\item\label{D5} Let $X(s,x,\xi)$ and $\Xi(s,x,\xi)$ be the
  Hamiltonian flow associated to $h$.  That is $X$ and $\Xi$ are
  solutions to the following ODEs:
   \[\left\{
     \begin{aligned}
       &\frac{d}{dt}X_j(s,x,\xi)=2a_{jk}\big(X(s,x,\xi)\big)\Xi_k(s,x,\xi) \\
       &\frac{d}{dt}\Xi_j(s,x,\xi)=-\frac{\del a_{lk}}{\del x_j}\big(X(s,x,\xi)\big)\Xi_l(s,x,\xi)\Xi_k(s,x,\xi)\\
       & X(0,x,\xi)=x \quad\text{and}\quad\Xi(0,x,\xi)=\xi.
     \end{aligned}\right.\]
   Then for each pair $x,\xi$ with $\xi\neq 0$ we assume that the sets
   $\{X(s, x, \xi) \mid s>0\}$ and $\{X(s,x,\xi) \mid s<0\}$ are
   unbounded.
\end{enumerate}

\begin{lem}[S.~Doi \cite{SD1996}]\label{doilemma}
  With $a_{jk}(x)$ satisfy \ref{D1}-\ref{D5}, there exist a symbol $p\in S^0_{1,0}$, with
  semi-norms bounded in terms of $C$, and a constant $B\in (0,1)$
  depending on $C$ and \ref{D5} such that
  $$H_hp\mathrel{\mathop:}=\{h,p\}=\sum_{i=1}^n\frac{\del h}{\del
    \xi_i}\frac{\del p}{\del x_i} - \frac{\del h}{\del x_i}\frac{\del
    p}{\del \xi_i}\geq \frac{B\abs{\xi}}{\jap{x}^2} -
    \frac{1}{B}\quad \text{for all }x, \xi \in \R^n.$$
\end{lem}

It is worth noting that in the case that the coefficients $a_{jk}(x)$
are elliptic, one can use the fact that the symbol $h$ is preserved
under the Hamiltonian flow together with the ellipticity to deduce
that $C^{-2}\abs{\xi}^2 \leq \abs{\Xi(s,x,\xi)}^2\leq
C^2\abs{\xi}^2$.  This implies that the solutions $X$ and $\Xi$ exist
for all times.  

For our purposes we need a version of Doi's lemma that is not centered
at the origin.  As before we let $\R^n=\bigcup_{\mu\in\Z^n}Q_\mu$
with $Q_\mu$ unit cubes with vertices in the lattice $\Z^n$ (indexed
by some corner), and let $x_\mu$ be the center of $Q_\mu$.

\begin{lem}\label{uncentereddoi1}
  Suppose $a_{jk}$ satisfies \emph{\ref{D1}-\ref{D5}}.  Then there exists
  a symbol $p_\mu\in S^0_{1,0}$ such that $H_hp_\mu(x,\xi)\geq
  C_1\frac{\abs{\xi}}{\jap{x-x_\mu}^2}-C_2$, where $C_1$, $C_2$ and the
  semi-norms of $p_\mu$ can be bounded independently of $\mu$.
\end{lem}
\begin{proof}
  For $\abs{x_\mu}<10$ we take $p_{\mu}=p$.  The content of the lemma
  is for $\abs{x_\mu}>>0$.  Let $\tilde h(x,\xi)=\abs{\xi}^2$,
  applying Lemma~\ref{doilemma} to the Laplacian can find a symbol $r(x,\xi)$
  so that $H_{\tilde h}r(x,\xi)\geq \tilde C_1
  \frac{\abs{\xi}}{\jap{x}^2}-\tilde C_2$.

  We take $p_\mu(x,\xi)=Np(x,\xi)+r(x-x_\mu,\xi),$ with $N$ to be
  determined.  Let $r_\mu(x,\xi)=r(x-x_\mu,\xi)$.  We calculate
  \[\begin{split}
   H_hr_\mu(x,\xi)&=\sum_{i=1}^n (2a_{ik}(x)\xi_k) \frac{\del r_\mu}{\del x_i}
     + \frac{\del a_{jk}}{\del x_i}(x)\xi_j\xi_k\frac{\del r_\mu}{\del \xi_i}\\
     &= \sum_{i=1}^n2\xi_i\frac{\del r_\mu}{\del x_i} +
     2(a_{ik}(x)-\delta_{ik})\xi_k \frac{\del r_\mu}{\del x_i}
     + \frac{\del a_{jk}}{\del x_i}(x)\xi_j\xi_k\frac{\del r_\mu}{\del \xi_i}.
   \end{split}\]
   For the second term  we have
   that $\abs{2(a_{ik}-\delta_{ik})\xi_k\frac{\del r_\mu}{\del
       x_i}}\leq C\frac{\abs{\xi}}{\jap{x}^2}$ from
     \ref{D4} and
  the bounds for the semi-norms of  $r_\mu$.  Similarly for the third
  term we have that \\
  $\abs{\frac{\del a_{jk}}{\del x_i}\xi_j\xi_k\frac{\del r_\mu}{\del \xi_i}} \leq
  \frac{C\abs{\xi}}{\jap{x}^2}.$
  The first term give that $H_{\tilde h}r_\mu \geq
  \tilde C_1\frac{\abs{\xi}}{\jap{x-x_\mu}^2}-\tilde C_2$.
  Finally,
   \[H_h(p_\mu)=NH_h(p)+H_h(r_\mu)\geq
   NC_1\frac{\abs{\xi}}{\jap{x}^2}-C\frac{\abs{\xi}}{\jap{x}^2} +
   \tilde C_1\frac{\abs{\xi}}{\jap{x-x_\mu}^2}-NC_2-\tilde C_2.\]
  So we choose $N$ large enough so that
  $NC_1\frac{\abs{\xi}}{\jap{x}^2}-C\frac{\abs{\xi}}{\jap{x}^2}\geq 0$
  and we get \[H_h(p_\mu)\geq C_1\frac{\abs{\xi}}{\jap{x-x_\mu}^2}-C_2.\]
\end{proof}

\begin{rem}\label{doibump}
  We remark that, perhaps by increasing our choice of $N$ to $2N$, we
  can ensure that $H_hp_\mu\geq C_1\frac{\abs{\xi}}{\jap{x-x_\mu}^2} +
  C_2\frac{\abs{\xi}}{\jap{x}^2}-C_3$.  This will be important to us
  when we want to consider linear estimates where the coefficients
  of the second order term depend on time.
\end{rem}


\section{Linear Results}\label{sec:linear}

In this section we consider the system
\begin{equation}\label{LinGenS}
\left\{\begin{aligned}
&\begin{aligned}\del_t u =& -\epsilon\Delta^2u+i\del_{x_j}a_{jk}(x,t)\del_{x_k}u + b_1(x,t,D)u +
\vec{b}_2(x,t)\cdot\grad \bar u\\  &+c_1(x,t,D)u+c_2(x,t,D)\bar u +f(x,t),\end{aligned}\\
&u(x,0)=u_0(x).
\end{aligned}\right.
\end{equation}

First we set some notation.  We denote by $A(x,t)$ the matrix
$(a_{jk}(x,t))_{j,k=1}^n$ and the symbol
$h(x,t,\xi)=\sum_{j,k=1}^na_{jk}(x,t)\xi_j\xi_k$.  For a function
$u(x,t)$ the Fourier transform of $u$ in the $x$ variable will be
denoted by $\hat u(\xi,t)$.  For a time varying symbol $q(x,t,\xi)$ we
use the following notation \[\Psi_qu(x,t) = \frac{1}{(2\pi)^n}\int
e^{ix\cdot\xi}q(x,t,\xi)\hat{u}(\xi,t)\,d\xi.\]
We let $\R^n=\bigcup_{\mu\in\Z^n}Q_\mu$
with $Q_\mu$ unit cubes with vertices in the lattice $\Z^n$.  We let
$x_\mu$ denote the the center of $Q_\mu$ and $Q_\mu^*$ denote its
concentric double.

When we use the linear estimates in the non-linear problem we will
evalunate our coefficients at some local solution.  For this reason
it will be important for the constant appearing in our final
inequality to depend only on the coefficients at $t=0.$ We therefore
take the convention that constants related to our coefficients at
$t=0$ will be denoted by $C_0$ and constants depending on our
coefficients at times other then 0 will be generically denoted by $C.$

We place the following assumptions on the coefficients. 

\begin{enumerate}[labelindent=0in, leftmargin=*, label=(L\arabic*)]

\item\label{LRegularity} There exist $M_L=M_L(n)\in\N$, and $C>0$ so
  that $a_{jk}(\cdot,t)$, $b_{2,j}(\cdot,t)$, $\del_ta_{jk}(\cdot,t)$,
  $\del_t b_{2,j}(\cdot,t) \in C_b^{M_L}(\R^n)$ for $j, k = 1,
  2,\ldots, n,$ with norm controlled by $C$.  We assume that,
  uniformly in $t$, the symbols $c_1(x,t,\xi)$, $c_2(x,t,\xi)\in
  S^0_{1,0}$ and $b_1(x,t,\xi)\in S^1_{1,0}$ with seminorms controlled
  by $C$.  In addition, we have that the norms of $a_{jk}(x,0)$
  $b_{2,j}(x,0)$ for $j,k=1, \ldots, n$ in $C^{M_L}_b$ together with
  the seminorms of $b_1(x,0,\xi), c_1(x,0,\xi)$ and $c_2(x,0,\xi)$ are
  controlled by $C_0$.

\item\label{LElliptic} The matrix
  $A(x,t)=\big(a_{jk}(x,t)\big)_{j,k=1\ldots n}$ has real valued
  entries, is symmetric, and is uniformly elliptic.
  That is, there is a positive number $C$ so that $$C^{-1} |\xi|^2 \leq
  \sum_{j, k = 1}^n a_{jk}(x,t)\xi_j\xi_k \leq C|\xi|^2.$$  
  Further at $t=0$ we have 
  $$C_0^{-1} |\xi|^2 \leq \sum_{j, k = 1}^n a_{jk}(x,0)\xi_j\xi_k \leq C_0|\xi|^2.$$

\item\label{LFlat} The matrix $A(x,t)$ is asymptotically flat. That is,
\[\abs{I-A(x,t)} + \abs{\grad_x A(x,t)} + \abs{\del_ta_{jk}(x,t)} +
     \abs{\del_t\del_{x_i}a_{jk}(x,t)} \leq \frac{C}{\jap{x}^2} \] and
\[\abs{I-A(x,0)} + \abs{\grad_x A(x,0)} + \abs{\del_ta_{jk}(x,0)} +
     \abs{\del_t\del_{x_i}a_{jk}(x,0)} \leq \frac{C_0}{\jap{x}^2}.\]

\item\label{LFirstOrder} The symbol $b_1(x,t,\xi)$ satisfies an
  estimate of the form \[\abs{\Re b_1(x,0,\xi)}\leq \sum_{\mu\in\Z^n}
  \beta_\mu^0\varphi_\mu(x)\abs{\xi},\] where $\beta_\mu^0\geq 0$,
  $\sum_{\mu\in\Z^n}\beta_\mu^0\leq C_0$, $\|\varphi_\mu\|_{C^{M_L}}\leq
  1$ and $\supp \varphi_\mu\subseteq Q_\mu^*$.  Also assume that
  \[\del_t \paren{\Re b_1(x,t,\xi)} = \sum_{\mu\in Z^n}
  \tilde\beta_\mu(t)\varphi_\mu(x,t,\xi),\] where
  $\tilde\beta_\mu(t)\geq 0$, and
  $\sum_{\mu\in\Z^n}\tilde\beta_\mu(t)\leq C$.  For all $t\in \R_+$
  the time varying symbols $\varphi_\mu(x,t,\xi)\in S^1_{1,0}$ with
  seminorms bounded by 1, independently of $t$ and
  $\mu$, and $\supp\varphi_\mu(\cdot,t,\xi)\subseteq Q_\mu^*$.

\item\label{LNoTrap} We assume that Hamiltonian flow associated to
  $h_0(x,\xi):=h(x,0,\xi)$ is non-trapping.  Let $p_\mu(x,\xi)$ be the
  Doi symbol for cube $Q_\mu$ associated to as constructed in the
  previous section.  We assume that these symbols satisfy
  \[H_{h_0}p_\mu \geq
  \frac{1}{C_0}\paren{\frac{\abs{\xi}}{\jap{x}^2} + %
  \frac{\abs{\xi}}{\jap{x-x_\mu}^2}}-C_0.\]
  The bounds in our arguments also depend on a finite number of
  seminorms of $p_\mu$ in $S^0_{1,0}$ and we assume these seminorms
  are controlled by $C_0$.  See Remark \ref{doibump}  in Section
  \ref{chap:doi} for this version of Doi's Lemma.
\end{enumerate}

\begin{thm}\label{aplin}
  Suppose that $u_0\in L^2$ and that there is a solution $u(x,t)$ to
  \eqref{LinGenS} in $C([0,T];L^2)$, where the coefficients satisfy
  \linref{LRegularity}-\linref{LNoTrap}.  Then there exist real
  numbers $T=T(C,C_0,\paren{\beta_\mu^0}_{\mu\in\Z^n})$ and
  $A=A(C_0,\paren{\beta_\mu^0}_{\mu\in\Z^n})$ so that if $f\in
  L^1([0,T];L^2)$, then $u$ satisfies
  \[\sup_{0\leq t\leq T}\|u(\cdot,t)\|_2^2+\sup_{\mu\in\Z^n}
  \|J^{1/2}u\|_{L^2(Q_\mu\times [0,T])}^2 \leq
    A\paren{\|u_0\|_2^2+\paren{\int_0^T\|f(\cdot,t)\|_2\,dt}^2}.\]
\end{thm}
\begin{proof}

We break the proof of this theorem into several steps.

\noindent\underline{Step 1.  Reduction to a system.}

Let $\vec w =
\begin{pmatrix}
  u\\ \bar u
\end{pmatrix}$, $\vec f = \begin{pmatrix}
  f\\ \bar f
\end{pmatrix}$, and $\vec{w}_0=\begin{pmatrix}
  u_0\\ \bar u_0
\end{pmatrix}.$
Let $\scr{L}(x,t)$ denote the operator $\del_{x_j}(a_{jk}(x,t)\del_{x_k}\cdot)$.
Then using the equations for $u$ and $\bar u$, we see that $w$ satisfies

\begin{equation}
\left\{\begin{aligned}
&\del_t \vec{w} = -\epsilon\Delta^2 I \vec{w}+\paren{iH+B+C}\vec{w} +
\vec{f}(x,t),\\
&\vec{w}(x,0)=\vec{w}_0(x),\end{aligned}\right.
\end{equation}
where \[H=
\begin{pmatrix}
  \scr{L}(x,t) & 0\\
  0 & -\scr{L}(x,t)
\end{pmatrix},\quad B=
\begin{pmatrix}
  B_{11} & B_{12} \\
  B_{21} & B_{22}
\end{pmatrix}:=
\begin{pmatrix}
  b_1(x,t,D) &  b_2(x,t)\cdot\grad \\
  \overline{b_2(x,t)}\cdot \grad & \overline{b_1(x,t,D)}
\end{pmatrix}\] \[\text{and }C=
\begin{pmatrix}
  c_1(x,t,D) & c_2(x,t,D) \\
  \overline{c_2(x,t,D)} & \overline{c_1(x,t,D)}
\end{pmatrix}.\]

Note that for the rest of this chapter $\ip{\vec u}{\vec v}=\int
u_1\bar v_1+u_2\bar v_2\,dx$ and $\norm{\vec u}_2^2=\ip{\vec u}{\vec u}$.

\noindent \underline{Step 2.  Diagonalize the first order terms.}

We now define an operator
$S= \begin{pmatrix}
  0 & s_{12} \\
  s_{21} & 0
\end{pmatrix}$ where $s_{12}$ and $s_{21}$ will be defined to be time
varying $\Psi$DO's and have symbols in $S_{1,0}^{-1}$ uniformly in
$t$.  We being by choosing $\phi \in C_0^\infty(\R^n)$ so that $\phi(y)=1$ for
$\abs{y}<1$ and $\phi(y)=0$ for $\abs{y}\geq 2$.  Let
$\theta_R(\xi)=1-\phi(\xi/R)$ and $\theta(\xi)=\theta_1(\xi).$

Let $\tilde h(x,t,\xi)=\theta_R(\xi)h^{-1}(x,t,\xi).$ Notice that, by
ellipticity \mbox{$h(x,t,\xi)\geq C^{-1}\abs{\xi}^2$}, hence
$\tilde h$ is a smooth function.  Let $\tilde{\scr{L}}=\Psi_{\tilde
  h}$, then we see that $\tilde{\scr{L}}\scr{L}=I+\Psi_{r_1}$ with
$r_1\in S_{1,0}^{-1}$ uniformly in $t$.

We define $S_{12}=\frac{1}{2}iB_{12}\tilde{\scr{L}}$ and
$S_{21}=-\frac{1}{2}iB_{21}\tilde{\scr{L}}$.  We denote the symbols of
$S_{12}$ and $S_{21}$ by $s_{12}(x,t,\xi)$ and $s_{21}(x,t,\xi)$
respectively.  Clearly $s_{ij}(x,t,\xi)\in S_{1,0}^{-1}$ uniformly in
$t$.  Let $\Lambda = I-S$, if we choose $R$ large enough, then we can
control the norms of $\Lambda$ and $\Lambda^{-1}$ by constants (see
Kenig's Park City Lecture 2 \cite{CK2005}).

We will use $\Lambda$ to change variables, and the resulting system
will have diagonal first order terms.  We first perform some
calculations that are necessary to rewrite the system in terms of
$\Lambda w$.
\[\begin{split}&-iH\Lambda +i\Lambda H=
  -iHI+iHS+iIH-iSH=-i\paren{SH-HS}\\=&
-i\paren{
  \begin{pmatrix}
    0 & S_{12}\\
    S_{21} & 0
  \end{pmatrix}
  \begin{pmatrix}
    \scr{L} & 0 \\
    0  & -\scr{L}
  \end{pmatrix} -
  \begin{pmatrix}
    \scr{L}  & 0 \\
   0 & -\scr{L}
  \end{pmatrix}
  \begin{pmatrix}
    0 & S_{12}\\
    S_{21} & 0
  \end{pmatrix}
}\\=&
\begin{pmatrix}
  0 & iS_{12}\scr{L}+i\scr{L}S_{12} \\
  -iS_{21}\scr{L}-i\scr{L}S_{21} & 0
\end{pmatrix}.
\end{split}
\]

Notice that $\scr{L}S_{12}=S_{12}\scr{L}+E^0_1$, where $E^0_1$ is an
error of order $0$.  Similarly $\scr{L}S_{21}=S_{21}\scr{L}+E_2^0$
with $E_2^0$ of order $0.$

Hence
$iS_{12}\scr{L}+i\scr{L}S_{12}=2iS_{12}\scr{L}+iE_1^0=-B_{12}\tilde{\scr{L}}\scr{L}+iE_1^0=-B_{12}+E_3^0$
and
$-iS_{21}\scr{L}-i\scr{L}S_{21}=-2iS_{21}\scr{L}-iE_2^0=-B_{21}\tilde{\scr{L}}\scr{L}+iE_2^0=-B_{21}+E_4^0$
with $E_3^0$ and $E_4^0$ errors of order $0$.

We write $B=B_d+B_{ad}$ where $B_d=
\paren{\begin{smallmatrix}
  B_{11} & 0 \\
  0  & B_{22}
\end{smallmatrix}}
$ and $B_{ad}=\paren{
\begin{smallmatrix}
  0 & B_{12} \\
  B_{21} & 0
\end{smallmatrix}}.
$  Now, \[\Lambda B_{ad}=IB_{ad}-SB_{ad}=B_{ad}-
\begin{pmatrix}
  S_{12}B_{21} & 0 \\
  0 & S_{21}B_{12}
\end{pmatrix}
=B_{ad}+E^0_{ad}\]
with $E_{ad}^0$ of order 0.

For the other terms we will want to commute $\Lambda$ and our
operators, in order to derive the equation for $\Lambda \vec w$.

Starting with
$\Lambda\del_t=\del_t\Lambda+[\del_t,\Lambda]=\del_t\Lambda+\del_t S$,
where this last expression is the matrix of $\Psi$DO's whose symbols
are given by $\del_t s_{jk}(x,t,\xi)$.  Using the bounds for $\del_t
b_2(x,t,\xi)$ and $\del_ta_{jk}(x,t,\xi)$ we again see that these
symbols are uniformly in $S_{1,0}^{-1}$.

Notice that $\Lambda B_d=B_d-SB_d$ and $B_d=B_d\Lambda+B_dS$, so that
$\Lambda B_d=B_d\Lambda+B_dS-SB_d=B_d\Lambda+E_d^0$ with $E_d^0$ is of
order 0.  \[\Lambda \epsilon I\Delta^2=\epsilon \Delta^2\Lambda -
\epsilon
\begin{pmatrix}
  0 & \Delta^2S_{12}-S_{12}\Delta^2 \\
  \Delta^2S_{21}-S_{21}\Delta^2 & 0
\end{pmatrix}
=\epsilon \Delta^2 I \Lambda + \epsilon \tilde{R}
\]
where $\tilde{R}$ is a matrix whose entries are operators of order 2.

We set $R=\tilde{R}\Lambda^{-1}$, which is still of order 2.  We write
\[\Lambda C + \del_tS+E^0_{ad}+E^0_d= \paren{\Lambda C \Lambda^{-1} + \del_tS\Lambda^{-1}+E^0_{ad}\Lambda^{-1}+E^0_d\Lambda^{-1}}\Lambda =:
\tilde{C}\Lambda\] with $\tilde{C}$ of order 0.

Lastly set $\vec F = \Lambda \vec f$. Define $\vec z=\Lambda \vec
w$ and apply $\Lambda$ to our equation.  We have

\[\Lambda \del_t\vec w=i\epsilon\Lambda\Delta^2 I \vec w
+\paren{i\Lambda H +\Lambda B +\Lambda C}\vec w+\Lambda \vec f\]

Using our calculations above we have

\[\del_t \vec z=-\epsilon \Delta^2 I \vec z + \epsilon R\vec z+ iH\vec z -
\begin{pmatrix}
  0 & B_{12}\\
  B_{21}& 0
\end{pmatrix}\vec w +B_d \vec z+B_{ad}\vec w+ \tilde{C}\vec{z}+\vec F.
\]
Hence if we set $\vec z_0=\Lambda \vec w_0$ to arrive at a system with
diagonal first order terms, namely

\begin{equation*}
\left\{\begin{aligned}
&\del_t \vec{z} = -\epsilon\Delta^2 I \vec{z}+\epsilon R\vec z + i H \vec z
+B_d\vec z +\tilde{C}\vec z +\vec F, \\
&\vec{z}(x,0)=\vec{z}_0(x).\end{aligned}\right.
\end{equation*}

As we pointed out earlier we have control of the norms of $\Lambda$
and $\Lambda^{-1},$ so deriving our desired estimates for $\vec z$
will imply the estimates for $\vec w$.

Since we work in slightly unusual norms it seems a good time to recall
them and justify this last statement.

\begin{mydef}
  Let $\R^n=\bigcup_{\mu\in \Z^n}Q_\mu$ as usual.  Let  $f:\R^n\times
  \R\to \C$ be measurable function.  We define
   \[\tnorm{f}_T=\sup_{\mu\in \Z^n}\norm{f}_{L^2(Q_{\mu}\times [0,T])}\]
and
\[\tnorm{f}'_T=\sum_{\mu\in\Z^n}\norm{f}_{L^2(Q_\mu\times[0,T])}.\]
\end{mydef}

\begin{thm}
  For $a\in S_{1,0}^0$ there is an $N(n)$ so that
  $\tnorm{\Psi_af}_T\leq C\tnorm{f}_{T}$ and $\tnorm{\Psi_af}'_{T}\leq
  C \tnorm{f}'_{T}$
\end{thm}

\begin{proof}
  See Kenig's Park City Lecture notes, Lecture 2 \cite{CK2005}.
\end{proof}

Now we return to our linear estimates.
It is important for the non-linear theory that we only make our
non-trapping assumptions at $t=0.$  The following lemma allows us to
handle time varying leading coefficients.

\begin{lem}\label{timedepdoi}
  Let $h(x,\xi)=a_{jk}(x,0)\xi_j\xi_k$ and let $p_\mu$ be the Doi
  symbol corresponding to $h$ centered at cube $Q_\mu$.  Then there
  exists a $T_1=T_1(C,C_0)$ so that uniformly for all $t<T_1$ the time varying
  symbol $h_t(x,\xi)=a_{jk}(x,t)\xi_j\xi_k$ satisfies
  \[H_{h_t}p_\mu(x,\xi)\geq \frac{1}{C_0}\frac{\abs{\xi}}{\jap{x-x_\mu}^2}-C_0.\]
\end{lem}
\begin{proof}
  By direct calculation we have
\[
\begin{aligned}
H_{h_t}p_\mu&=\sum_{i=1}^n \diff{h}{\xi_i}\diff{p_\mu}{x_i} -
\diff{h}{x_i}\diff{p_\mu}{\xi_i}+ \paren{\diff{h_t}{\xi_i} -
  \diff{h}{\xi_i}}\diff{p_\mu}{x_i}
- \paren{\diff{h_t}{x_i}-\diff{h}{x_i}}\diff{p_\mu}{\xi_i}\\
&=
  H_hp_\mu+
  \sum_{i=1}^n 2\paren{a_{ik}(x,t)-a_{ik}(x,0)}\xi_k\diff{p_\mu}{x_i}\\
  &\qquad\qquad\qquad\qquad-\paren{\diff{a_{jk}(x,t)}{x_i}
  -\diff{a_{jk}(x,0)}{x_i}}\xi_j\xi_k\diff{p_\mu}{\xi_i}.
\end{aligned}
\]

From the asymptotic flatness condition \linref{LFlat} there is a
$T_1=T_1(C,C_0)$ so that if $t<T_1$ we have that we have that
\[
\abs{\sum_{i=1}^n 2\paren{a_{ik}(x,t)-a_{ik}(x,0)}\xi_k\diff{p_\mu}{x_i}
 - \paren{\diff{a_{jk}(x,t)}{x_i}-\diff{a_{jk}(x,0)}{x_i}}
 \xi_j\xi_k\diff{p_\mu}{\xi_i}}\leq \frac{1}{C_0}\frac{\abs{\xi}}{\jap{x}^2}.
\]

Now using \linref{LNoTrap} we get that
\[H_{h_t}p_\mu\geq H_h p_\mu-\frac{1}{C_0}\frac{\abs{\xi}}{\jap{x}^2}\geq
\frac{1}{C_0}\frac{\abs{\xi}}{\jap{x-x_\mu}^2}-C_0.\]
\end{proof}

\noindent
\underline{Step 3. Energy Estimates.}

The goal of this section is to conclude the proof.  The program is to
again introduce an invertible change of variables, this time based on
Doi's Lemma.  It is Doi's lemma that allows us to absorb the first
order terms.

Set \[\Psi_M=
\begin{pmatrix}
  \Psi_{q_1} & 0 \\
  0 & \Psi_{q_2}
\end{pmatrix}
\] where $\Psi_{q_1}$, $\Psi_{q_2}$ are invertible $\Psi$DO's of order 0 that
will be defined below.

First we compute the necessary commutators that arise in the change of
variables.
For the leading order terms
\[\Psi_M\epsilon\Delta^2I = \epsilon\Delta^2I\Psi_M\ +
\epsilon \tilde R^3\Psi_M\] with $\tilde R^3$ of order 3.  The second
order remainder term yields.  $\Psi_M\epsilon R=\epsilon R
\Psi_M+\epsilon \tilde R^1\Psi_M$ We collect these remainder terms by
setting $R^3=\tilde R^3+R+\tilde R^1$, which is of order 3.  The
remaining terms pose no difficulty.  The first and zeroth order terms
simply give $\Psi_M B_d=B_d\Psi_M+E^0_5$, $\Psi_M\tilde C=\tilde
C\Psi_M+E^0_6$ and lastly we set $\vec G=\Psi_M\vec F$.

Again we absorb the error terms of order 0 into $\tilde C$.  By
setting $\vec \alpha = \Psi_M\vec z$ and $\vec \alpha_0=\Psi_M\vec
z_0$ we arrive at the system

\begin{equation}
\left\{\begin{aligned}
&\del_t \vec{\alpha} = -\epsilon\Delta^2 I \vec{\alpha}+\epsilon
R^3\alpha - i [H,\Psi_M] \vec z
+iH\vec \alpha +B_d\vec \alpha +\tilde{C}\vec \alpha +\vec G \\
&\vec{\alpha}(x,0)=\vec{\alpha}_0(x).\end{aligned}\right.
\end{equation}

To construct $\Psi_M$ we let $\R^n=\bigcup_{\mu\in\Z^n}Q_\mu$ as
usual.  Fix a cube $Q_{\mu_0}$ and let
\[\gamma_{\mu_0}(x,\xi)=p_{\mu_0}(x,\xi)+\sum_{\mu\in\Z^n}\beta_\mu^0 p_\mu(x,\xi)\]
with $\beta_\mu^0$ as in \linref{LFirstOrder}.  Notice that
$\gamma_{\mu_0}\in S_{1,0}^0$ with seminorms controlled in terms of $C_0$.

Let $q_1(x,\xi)=\exp(\theta_R(\xi)\tilde{C}_0\gamma_{\mu_0}(x,\xi))$ and
$q_2(x,\xi) = \exp(-\theta_R(\xi)\tilde{C}_0gamma_{\mu_0}(x,\xi)).$
Where $\tilde{C}_0$ depends on $C_0$ will be chosen below.  Notice that
again, if we take $R$ large we may ensure that $\Psi_M$ is
invertible uniformly in $\mu_0$.   We now
compute  \[-i[H,\Psi_M] = -i
\begin{pmatrix}
  \scr{L}\Psi_{q_1}-\Psi_{q_1}\scr{L} & 0 \\
  0 & -\scr{L}\Psi_{q_2}+\Psi_{q_2}\scr{L}
\end{pmatrix}.
\]

Let $\ell(x,t,\xi)$ be the symbol for $\scr{L}$, then
$\ell(x,t,\xi) = a_{jk}(x,t)\xi_j\xi_k + \del_{x_j}a_{jk}(x,t)\xi_k =
h_t(x,\xi) + \ell_1(x,t,\xi)$.
Note that $\{\ell_1,q_i\}\in S^0_{1,0}$ uniformly in $t$ for $i=1,2.$  Hence,
\[-i\paren{\scr{L}\Psi_{q_1}-\Psi_{q_1}\scr{L}}=\Psi_{\{h_t,q_1\}}+E^0_7,\]
with $E^0_7$ an operator of order 0.

It follows that \[\{h_t,q_1\}=\paren{\diff{h_t}{\xi_i}\theta_R(\xi)
  \diff{\gamma_{\mu_0}}{x_i} -
  \diff{h_t}{x_i}\theta_R(\xi)\diff{\gamma_{\mu_0}}{\xi_i} - \diff{h_t}{x_i}\diff{\theta_R}{\xi_i}\gamma_{\mu_0}}
e^{\theta_R(\xi)\gamma_{\mu_0}},\]
where the last term in the parentheses is in $S^{-\infty}_{1,0}$.
Therefore \[-i\paren{\scr{L}\Psi_{q_1}-\Psi_{q_1}\scr{L}} =
-\Psi_{\theta_RH_{h_t}\gamma_{\mu_o}}\Psi_{q_1} + E^0_8,\]  with
$E_8^0$ of order 0.  In the same way
\[-i\paren{-\scr{L}\Psi_{q_2}+\Psi_{q_2}\scr{L}} =
-\Psi_{\theta_RH_{h_t}\gamma_{\mu_o}}\Psi_{q_2} + E^0_9.\]

Thus our system (after absorbing errors into $\tilde C$) looks like

\begin{equation*}
\left\{\begin{aligned}
&\begin{aligned}\del_t \vec{\alpha} =& -\epsilon\Delta^2 I \vec{\alpha}+\epsilon
R^3\alpha +
\begin{pmatrix}
  -\Psi_{\theta_R(\xi)H_{h_t}\gamma_{\mu_0}} & 0\\
 0 &  -\Psi_{\theta_R(\xi)H_{h_t}\gamma_{\mu_0}}
\end{pmatrix}
\vec \alpha \\
&\qquad +iH\vec \alpha +B_d\vec \alpha +\tilde{C}\vec \alpha +\vec G,
\end{aligned} \\
&\vec{\alpha}(x,0)=\vec{\alpha}_0(x).\end{aligned}\right.
\end{equation*}

We now proceed to derive energy estimates for $\alpha.$ Consider
\[
\begin{aligned}
\del_t\ip{\vec\alpha}{\vec\alpha} &= \ip{\del_t\vec\alpha}{\vec\alpha}
+\ip{\vec\alpha}{\del_t\vec\alpha}
\\&=\ip{-\epsilon\Delta^2I\vec\alpha}{\vec\alpha} +
\ip{\vec\alpha}{-\epsilon\Delta^2I\vec\alpha} + \ip{\epsilon
  R^3\vec\alpha}{\vec\alpha} + \ip{\vec\alpha}{\epsilon R^3\vec\alpha}
+ \\ &\ip{iH\vec\alpha}{\vec\alpha} +\ip{\vec\alpha}{iH\vec\alpha} +
\ip{B_d\vec\alpha}{\vec\alpha} + \ip{\vec\alpha}{B_d\vec\alpha} +\\&
\ip{C\vec\alpha}{\vec\alpha} + \ip{\vec\alpha}{C\vec\alpha} +
\ip{\vec G}{\vec\alpha} + \ip{\vec\alpha}{\vec G} + \\
& \ip{
\begin{pmatrix}
  -\Psi_{\theta_R(\xi)H_{h_t}\gamma_{\mu_0}} & 0\\
 0 &  -\Psi_{\theta_R(\xi)H_{h_t}\gamma_{\mu_0}}
\end{pmatrix}\vec\alpha}{\vec\alpha}
 +\\& \ip{\vec\alpha}{\begin{pmatrix}
  -\Psi_{\theta_R(\xi)H_{h_t}\gamma_{\mu_0}} & 0\\
 0 &  -\Psi_{\theta_R(\xi)H_{h_t}\gamma_{\mu_0}}
\end{pmatrix}\vec\alpha}.
\end{aligned}
\]

In the first two terms we have we have that
$-\epsilon\ip{\Delta^2I\vec\alpha}{\vec\alpha} -
\epsilon\ip{\vec\alpha}{\Delta^2I\vec\alpha} =
-2\epsilon\ip{\Delta I\vec\alpha}{\Delta I\vec\alpha} =
-2\epsilon\|\Delta\vec\alpha\|_2^2$.
\pagebreak[1]
The second two terms contribute
$\ip{\epsilon R^3\vec\alpha}{\vec\alpha} +
\ip{\vec\alpha}{\epsilon R^3\vec\alpha} =
2\epsilon\Re\ip{R^3\vec\alpha}{\vec\alpha} =
2\epsilon\Re\ip{J^{-3/2}IR^3\vec\alpha}{J^{3/2}I\vec\alpha}$.

As both $J^{3/2}$ and $J^{-3/2}IR^3$ are operators of order $3/2$ we
can bound this by $C\|\vec\alpha\|^2_{H^{3/2}}.$ Now by interpolation we
have that $\|\vec\alpha\|^2_{H^{3/2}}\leq \eta_0\norm{\Delta I\vec\alpha}_2^2
+ \frac{2}{\eta_0}\|\vec\alpha\|^2_2.$

Hence \[\abs{2\epsilon\Re\ip{R^3\vec\alpha}{\vec\alpha}}\leq
2\epsilon C\eta_0\|\Delta I\vec\alpha\|_2^2 +
\frac{4\epsilon C}{\eta_0}\|\vec\alpha\|_2^2.\]  By setting
$\eta_0=1/(2C)$ we can absorb the first term into
$-2\epsilon\|\Delta\vec\alpha\|^2_2$ to get the first four terms are bounded
by $-\epsilon \|\Delta I\vec\alpha\|^2_2+8\epsilon C^2\|\vec\alpha\|^2_2$.

We now turn our attention to first order terms.  That is, the last two
terms and the terms involving $B_d$.  Consider the matrix of symbols

\[F:=
\begin{pmatrix}
  -\theta_R(\xi)H_{h_t}\gamma_{\mu_0}(x,\xi)+b_1(x,t,\xi) & 0 \\
  0 &  -\theta_R(\xi)H_{h_t}\gamma_{\mu_0}(x,\xi)+\overline{b_1(x,t,-\xi)}
\end{pmatrix}.
\]

We will need to control $F+F^*$ to apply the vector valued G\aa rding's
 inequality. Let $\phi_{Q_\mu^*}$ be a smooth cut off to the
double of $Q_\mu$ and let $\chi_{Q_\mu^*}=\phi^2_{Q_\mu^*}$.  By our construction of $\gamma_{\mu_0}$ we have
that
\[
\begin{aligned}
-\tilde{C}_0\theta_R(\xi)H_{h_t}\gamma_{\mu_0}\leq&
\tilde{C}_0\theta_R(\xi)\Bigg(-\frac{1}{C_0}\frac{\abs{\xi}}{\jap{x-x_{\mu_0}}^2}+C_0
 \\& \qquad\qquad\qquad\qquad
-\sum_{\mu\in\Z^n}\beta_\mu^0\paren{\frac{1}{C_0}\frac{\abs{\xi}}{\jap{x-x_\mu}^2}-C_0}\Bigg)\\
\leq &\tilde{C}_0\theta_R(\xi)\paren{-C_0'\abs{\xi}\chi_{Q_{\mu_0}^*} -
  \sum_{\mu\in\Z^n}\beta_\mu^0 C_0'\abs{\xi}\chi_{Q_{\mu}^*}} + \tilde{C}_0C_0''\\
\leq &
\tilde{C}_0\theta_R(\xi)\paren{-C_0'\abs{\xi}\chi_{Q_{\mu_0}^*}-
C_0'\abs{\Re b_1(x,0,\xi)}} + \tilde{C}_0C_0''.
\end{aligned}
\]

Here we choose $\tilde C_0$ so that $\tilde C_0C_0'\geq 2$.
Now we have that
\begin{multline*}
  -\theta_R(\xi)H_{h_t}\gamma_{\mu_0}+2\Re b_1(x,t,\xi) =\\
   \paren{\theta_R(\xi)H_{h_t}\gamma_{\mu_0}+2\Re b_1(x,0,\xi)} +
  2\Re \int_0^t\del_tb_1(x,s,\xi)\,ds\\
  \leq -\tilde C_0'\theta_R(\xi)\abs{\xi}\chi_{Q_{\mu_0}^*}+\tilde C_0''+
  2\int_0^t\paren{\sum_{\mu\in
      Z^n} \tilde\beta_\mu(s)\varphi_\mu(x,s,\xi)}\,ds.
\end{multline*}
Let $p(x,t,\xi)= 2\int_0^t\paren{\sum_{\mu\in
      Z^n} \tilde\beta_\mu(s)\varphi_\mu(x,s,\xi)}\,ds.$
Apply the vector valued G\aa rding inequality (see
\cite{HK1981}, \cite{MT1991}) to get

\begin{multline*}
\Re\ip{
  \begin{pmatrix}
    -\Psi_{\theta_R(\xi)H_{h_t}\gamma_{\mu_0}}+b_1(x,t,D) & 0\\
     0 & -\Psi_{\theta_R(\xi)H_{h_t}\gamma_{\mu_0}}+\overline{b_1(x,t,D)}
  \end{pmatrix}
\vec\alpha}{\vec\alpha}\\
\leq C\norm{\vec\alpha}_2^2-\Re\ip{
  \begin{pmatrix}
    \Psi_{\theta_R \abs{\xi}\chi_{Q_{\mu_0}^*}} & 0 \\
    0 & \Psi_{\theta_R \abs{\xi}\chi_{Q_{\mu_0}^*}}
  \end{pmatrix}
\vec\alpha}{\vec\alpha}+\\\Re\ip{
  \begin{pmatrix}
    \Psi_{p(x,t,\xi)} & 0 \\
    0 & \Psi_{p(x,t,-\xi)}
  \end{pmatrix}
\vec\alpha}{\vec\alpha}.
\end{multline*}

We denote this last term by $\Re\ip{E_M\vec\alpha}{\vec\alpha}$. The
symbol of the  operator \\
$\Psi_{\theta_R(\xi)\abs{\xi}\chi_{Q_{\mu_0}^*}}-J^{1/2}\chi_{Q_{\mu_0}^*}J^{1/2}$
is in $S^0_{1,0}$.  Hence we have
\begin{multline*}\Re\ip{
  \begin{pmatrix}
    \Psi_{\tilde C_0\theta_R \abs{\xi}\chi_{Q_{\mu_0}^*}} & 0 \\
    0 & \Psi_{\tilde C_0 x\theta_R \abs{\xi}\chi_{Q_{\mu_0}^*}}
  \end{pmatrix}
  \vec\alpha}{\vec\alpha} \\ \geq
\ip{\phi_{Q_{\mu_0}^*}J^{1/2}I\vec\alpha}{\phi_{Q_{\mu_0}^*}J^{1/2}I\vec\alpha}
-C\norm{\vec\alpha}_{L^2}^2 \\\geq\;
\|\phi_{Q_{\mu_0}^*}J^{1/2}I\vec\alpha\|_{L^2(Q_{\mu_0})}-C\norm{\vec\alpha}_{L^2}^2.
\end{multline*}

To handle the terms involving $H$ notice that
$\int \scr{L}\alpha_1\alpha_1=\int\alpha_1\scr{L}\alpha_1$, and hence
$i\ip{H\vec\alpha}{\vec\alpha}-i\ip{\vec\alpha}{H\vec\alpha}=0.$
For the terms involving $\tilde C$ we use Cauchy-Schwartz inequality,
$\abs{\ip{\tilde C \vec\alpha}{\vec\alpha}}\leq
\norm{\tilde C\vec\alpha}_2\norm{\vec\alpha}_2\leq
C\norm{\vec\alpha}_2^2$.

Putting all this together we see that
$$\frac{d}{dt}\norm{\vec\alpha}_2^2\leq
-\epsilon\|\Delta I\vec\alpha\|_2^2+C\norm{\vec\alpha}_2^2-\norm{J^{1/2}I\vec\alpha}_{L^2(Q_{\mu_0})}^2 +
2\abs{\Re\ip{\vec\alpha}{\vec G}}+\Re\ip{E_M\vec\alpha}{\vec\alpha}.$$
Integrating in time we find that
\begin{multline*}\norm{\vec\alpha(t)}_2^2 +
\norm{\phi_{Q_{\mu_0}^*}J^{1/2}I\vec\alpha}_{L^2(Q_{\mu_0}\times [0,t])}^2\leq \norm{\vec\alpha(0)}_2^2
+C\int_0^t\norm{\vec\alpha}_2^2\,ds +\\ 2\int_0^t\abs{\Re\ip{\vec\alpha}{\vec G}}\,ds +
\int_0^t\Re\ip{E_M\vec\alpha}{\vec\alpha}\,ds.\end{multline*}

In order to handle the terms
$\int_0^t\Re\ip{E_M\vec\alpha}{\vec\alpha}\,ds$, we have that

\begin{multline*}
\int_0^t \int_0^s \int\sum_{\mu\in\Z^n}\tilde\beta_\mu(r)\Psi_{\varphi_\mu(x,r,\xi)}\alpha_1(s)\overline{\alpha_1(s)}\,dx\,dr\,ds=\\
\int_0^t \sum_{\mu\in\Z^n}\tilde\beta_\mu(r)
\int_r^t\int\Psi_{\varphi_\mu(x,r,\xi)}
\alpha_1(s)\overline{\alpha_1(s)}\,dx\,ds\,dr.
\end{multline*}
Our estimates on $\varphi_\mu(x,s,\xi)$ give us that
$\int\Psi_{\varphi_\mu(s,x,\xi)}\alpha_1\overline{\alpha_1}\,dx\leq
\norm{J^{1/2}\alpha_1}_{L^2(Q_\mu^*)}^2+C\norm{\alpha_1}_2^2$.
Thus we have $$\abs{\int_0^t \ip{E_M\vec\alpha}{\vec\alpha}}\leq
Ct\sup_{\mu\in \Z^n}\norm{J^{1/2}I\vec\alpha}_{L^2([0,t]\times
  Q_\mu)}+Ct\sup_{0\leq s \leq t}\norm{\vec\alpha}_2^2$$

Hence, after taking a supremum over $0\leq t \leq T$, we arrive at
\begin{multline*}\sup_{0\leq t\leq T}\norm{\vec\alpha(t)}_2^2 +
\norm{\phi_{Q_{\mu_0}^*}J^{1/2}I\vec\alpha}_{L^2(Q_{\mu_0}\times [0,T])}^2\leq
\norm{\vec\alpha(0)}_2^2 
+CT\sup_{0\leq t\leq T}\norm{\vec\alpha}_2^2 + \\
2\int_0^T\abs{\Re\ip{\vec\alpha}{\vec G}}\,ds + CT\sup_{\mu\in
  \Z^n}\norm{J^{1/2}I\vec\alpha}_{L^2([0,T]\times Q_\mu)}.
\end{multline*}
By choosing $T$ small we may make $CT\sup_{0\leq t\leq
  T}\norm{\vec\alpha}_2^2\leq \frac{1}{2} \sup_{0\leq t\leq
  T}\norm{\vec\alpha}_2^2$ and absorb this term into the left hand
side.  In this way we get

\begin{multline}\label{basicest}
\sup_{0\leq t\leq T}\norm{\vec\alpha(t)}_2^2 +
\norm{J^{1/2}I\vec\alpha}_{L^2(Q_{\mu_0}\times [0,T])}^2\leq
2\norm{\vec\alpha(0)}_2^2 + \\
4\int_0^T\abs{\Re\ip{\vec\alpha}{\vec G}}\,ds +
CT\sup_{\mu\in\Z^n}\norm{J^{1/2}I\vec\alpha}_{L^2([0,T]\times Q_\mu)}.
\end{multline}

In terms of $\vec z$ our estimates now tell us
\begin{multline*}\sup_{0\leq t\leq T}\norm{\Psi_M\vec z(t)}_2^2 +
\norm{\phi_{Q_{\mu_0}^*}J^{1/2}I\Psi_M\vec z}_{L^2(Q_{\mu_0}\times [0,T])}^2\leq
2\norm{\Psi_M\vec z(0)}_2^2 + \\
4\int_0^T\abs{\Re\ip{\Psi_M\vec z}{\vec G}}\,ds +
CT\sup_{\mu\in\Z^n}\norm{J^{1/2}I\Psi_M\vec z}_{L^2([0,T]\times
  Q_\mu)}.
\end{multline*}

But notice that,
$J^{1/2}I\Psi_M=\Phi_MJ^{1/2}I+E$,
where $E$ is of order 0.  Hence

\begin{multline*}\int_0^T \norm{\phi_{Q_{\mu_0}^*}J^{1/2}I\Psi_M\vec
  z}_{L^2(Q_{\mu_0})}^2\,dt\geq\\
  \int_0^TC_0\norm{J^{1/2}I\vec z}_{L^2(Q_{\mu_0})}^2\,dt - CT\sup_{0\leq t\leq T}\norm{\vec
  z}_2^2.\end{multline*}  Thus, possibly after another restriction in
$T$, we arrive at
\begin{multline*}\sup_{0\leq t\leq T}\norm{\vec z(t)}_2^2 +
\norm{J^{1/2}I\vec z}_{L^2(Q_{\mu_0}\times [0,T])}^2\leq
C_0\Bigg(\norm{\vec z(0)}_2^2 +
\int_0^T\abs{\Re\ip{\Psi_M\vec z}{\vec G}}\,ds \\
+ CT\sup_{\mu\in\Z^n}\norm{J^{1/2}I\vec z}_{L^2([0,T]\times
  Q_\mu)}\Bigg).
\end{multline*}

Now estimate the term involving $\vec G$.

\begin{multline*}\label{L1intest}
\int_0^T\abs{\Re\ip{\Psi_M\vec z}{\vec G}}\,ds\leq
\int_0^TC_0\norm{\vec z}_2\norm{\vec G}_2\,dt\leq
C_0\sup_{0\leq s\leq t}\norm{\vec z(s)}_2\norm{G}_{L^1_t L^2_x}\\\leq
C_0\eta\sup_{0\leq s \leq t}\norm{\vec z(t)}_2^2+\frac{C_0}{\eta}\norm{G}_{L^1_tL^2_x}^2
\end{multline*}

Choosing $\eta$ small enough to absorb the term
involving $z$ to the left hand side.  Our estimate now is of the form
\begin{multline*}\sup_{0\leq t\leq T}\norm{\vec z(t)}_2^2 +
\norm{J^{1/2}I\vec z}_{L^2(Q_{\mu_0}\times [0,T])}^2\leq
C_0\Bigg(\norm{\vec z(0)}_2^2 + \norm{G}^2_{L^1_tL^2_x}+  \\
CT\sup_{\mu\in\Z^n}\norm{J^{1/2}I\vec z}_{L^2([0,T]\times Q_\mu)}\Bigg).\end{multline*}

Finally to get Theorem~\ref{aplin} we take a supremum in $\mu_0$, then
after a suitable restriction in $T$ we may absorb
$CT\sup_{\mu\in\Z^n}\norm{J^{1/2}I\vec z}_{L^2([0,T]\times Q_\mu)}$
into the left hand side.  Keeping in mind that estimates for $z$ will
imply the corresponding estimates in $u.$
\end{proof}

We now turn to a perturbation result. It is possible to weaken
the non-trapping condition \linref{LNoTrap}.  It is enough to assume
that the second order coefficients are ``close'' to coefficients that
are non-trapping.

To this end, we again consider equation \ref{LinGenS}.  We still
assume that the coefficients satisfy conditions
\linref{LRegularity}--\linref{LFirstOrder}.  Instead of
\linref{LNoTrap}, suppose that $A(x,t)=A_0(x,t)+\eta A_1(x,t)$.
Assume that $h_0(x,\xi)=\ip{A_0(x,0)\xi}{\xi}$ satisfies the
non-trapping condition \linref{LNoTrap}.  In addition, assume that
$\abs{A_1(x,t)}+\abs{\grad_xA_1(x,t)}\leq\frac{C}{\jap{x}^2}$
uniformly in $t$.  Then for $\eta$ sufficiently small, depending on
$C$ and $C_0$,  the conclusion of Theorem~\ref{aplin} holds.

To see this, notice that we only use the non-trapping condition
\linref{LNoTrap} in the proof of Lemma~\ref{timedepdoi}.  We
will now prove this lemma in the under these slightly more general
assumptions.

\begin{lem}
  Let $h_0(x,\xi)$ be as above and let $p_\mu$ be the Doi
  symbol corresponding to $h_0$ centered at cube $Q_\mu$.  Then there
  exists a $T_1=T_1(C,C_0)$ so that, uniformly for all $t<T_1$, the time varying
  symbol $h_t(x,\xi)=a_{jk}(x,t)\xi_j\xi_k$ satisfies
  \[H_{h_t}p_\mu(x,\xi)\geq \frac{1}{C_0}\frac{\abs{\xi}}{\jap{x-x_\mu}^2}-C_0.\]
\end{lem}
\begin{proof}
  To facilitate calculations we use the following notations for the
  matrix entries $(a^0_{jk}(x,t))_{j,k=1\ldots n}:=A_0(x,t)$ and
  $(a^1_{jk}(x,t))_{j,k=1\ldots n}:=A_1(x,t)$.  It is also convenient
  to denote
  $k_0(x,t,\xi)=\ip{A_0(x,t)\xi}{\xi}$, and
  $k_1(x,t,\xi)=\ip{A_1(x,t)\xi}{\xi}$.  Proceeding with our
  calculation as before we have
\[
\begin{aligned}
H_{h_t}p_\mu&=\sum_{i=1}^n \diff{h_0}{\xi_i}\diff{p_\mu}{x_i} -
\diff{h_0}{x_i}\diff{p_\mu}{\xi_i}+ \paren{\diff{k_0}{\xi_i}
-\diff{h_0}{\xi_i}+\eta\diff{k_1}{\xi_1}}\diff{p_\mu}{x_i}\\
&\qquad\qquad\qquad\qquad
  - \paren{\diff{k_0}{x_i}-\diff{h_0}{x_i}
    +\eta\diff{k_1}{x_i}}\diff{p_\mu}{\xi_i}\\
&=
  H_{h_0}p_\mu+
  \sum_{i=1}^n 2\paren{a^0_{ik}(x,t)-a^0_{ik}(x,0)+\eta a^1_{ik}(x,t)}\xi_k\diff{p_\mu}{x_i}\\
  &\qquad\qquad\qquad\qquad-\paren{\diff{a^0_{jk}(x,t)}{x_i}
  -\diff{a^0_{jk}(x,0)}{x_i}+\eta\diff{a^1_{jk}(x,t)}{x_i}}\xi_j\xi_k\diff{p_\mu}{\xi_i}.
\end{aligned}
\]

As before the asymptotic flatness condition \linref{LFlat} there is a
$T_1=T_1(C,C_0)$ so that if $t<T_1$ we have that we have that
\begin{multline*}
\abs{\sum_{i=1}^n 2\paren{a^0_{ik}(x,t)-a^0_{ik}(x,0)}\xi_k\diff{p_\mu}{x_i}
\right.\\\left. - \paren{\diff{a^0_{jk}(x,t)}{x_i}-\diff{a^0_{jk}(x,0)}{x_i}}
 \xi_j\xi_k\diff{p_\mu}{\xi_i}}\leq \frac{1}{2C_0}\frac{\abs{\xi}}{\jap{x}^2}.
\end{multline*}

Our conditions on $A_1$, together with the control of the  seminorms
of $p_\mu$ give that

\[ \eta\abs{\sum_{i=1}^na^1_{ik}(x,t)\xi_k\diff{p_\mu}{x_i}+\diff{a^1_{jk}(x,0)}{x_i}
 \xi_j\xi_k\diff{p_\mu}{\xi_i}}\leq \frac{\eta C \abs{\xi}}{\jap{x}^2}\]

We choose $\eta$ so that $\frac{\eta C
  \abs{\xi}}{\jap{x}^2}\leq\frac{1}{2C_1}\frac{\abs{\xi}}{\jap{x}^2}$. Now
using the assumption that $h^0$ satisfies \linref{LNoTrap} we get that
\[H_{h_t}p_\mu\geq H_{h_0} p_\mu-\frac{1}{C_1}\frac{\abs{\xi}}{\jap{x}^2}\geq
\frac{1}{C_1}\frac{\abs{\xi}}{\jap{x-x_\mu}^2}-C_1.\]

Using the same version of Doi's lemma as before.
\end{proof}


\section{Nonlinear Results}\label{sec:nonlinear}

In this section we approach \eqref{theproblem} by the artificial
viscosity method.  Hence, we are interested in the system
\begin{equation}\label{NonLinGenS1}
\left\{\begin{aligned}
&
\begin{aligned}\del_t u =&-\epsilon\Delta^2u+ia_{jk}(\XuDu)\del_{x_j}\del_{x_k}u
\\& +   \vec b_1(\XuDu)\cdot\grad u
 +\vec{b}_2(\XuDu)\cdot\grad \bar u\\ & +c_1(\Xu)u +
c_2(\Xu)\bar u +f(x,t),
\end{aligned}\\
&u(x,0)=u_0(x).
\end{aligned}\right.
\end{equation}

We assume the coefficients satisfy the conditions set out in the
introduction.  We take $s>N+n+4$ with $N$ as in
\linref{LRegularity} and even.  We take $\tilde N> s+2$.

We abbreviate our system to
\begin{equation*}
\left\{
    \begin{aligned}
      &\del_t u=-\epsilon\Delta^2u+\scr{L}(u)u+f(x,t)\\
      &u(x,0)=u_0(x),
    \end{aligned}
\right.
\end{equation*}
where
\begin{multline*}\scr{L}(u)v=ia_{jk}(\XuDu)\del_{x_j}\del_{x_k}v +
\vec b_1(\XuDu)\cdot\grad v + \\\vec{b}_2(\XuDu)\cdot\grad \bar v +
c_1(\Xu)v + c_2(\Xu)\bar v.\end{multline*}

\begin{thm}\label{contractmapping}
  Take $s> n+3$.  For $v_0\in H^s$ and
$f\in L^\infty([0,1];H^s),$ define $\lambda=\|v_0\|_s +
\int_0^1\|f(t)\|_{H^s}\,dt$.  Define
$$X_{M_0,T}=\{v:\R^n\times[0,T]\to \C \mid v(x,0)=v_0, v\in
  C([0,T];H^s), \norm{v}_{L^\infty_tH^s_x}\leq M_0\}.$$  If
  $\lambda<M_0/2$, then there exists $T_\epsilon$,  $1>T_\epsilon>0$, so
that equation \eqref{NonLinGenS1} with initial data $v_0$ has a
unique solution $v^\epsilon\in X_{M_0,T^\epsilon}.$
\end{thm}

\begin{proof}

For $t<1$, consider the integral form the equation
\[\Gamma v(t)=e^{-\epsilon t\Delta^2}v_0 +
\int_0^t e^{-\epsilon (t-t')\Delta^2}\paren{\scr{L}(v)v(t')+f(\cdot,t')}\,dt'.\]

We show that $\Gamma$ is a contraction mapping on the space
$X_{M_0,T}$ after a suitable restriction of $T$.  So let $\alpha$ be a
multi-index such that $\abs{\alpha}=s$ and consider
\[\del_x^\alpha \Gamma v(t)=e^{-\epsilon t\Delta^2}\del_x^\alpha v_0 +
\int_0^t\del_x^\alpha \paren{e^{-\epsilon(t-t')\Delta^2}\scr{L}(v)v+f}(t')\,dt'.\]
Choose multi-indices $\beta$ and $\beta'$ so that $\abs{\beta'}=2$,
$\abs{\beta}=s-2$, and $\alpha=\beta+\beta'$.
\begin{multline*}
\del_x^\alpha \Gamma v(t)= e^{-\epsilon t\Delta^2}\del_x^\alpha v_0
+
\int_0^t\del_x^{\beta'}e^{-\epsilon(t-t')\Delta^2}
\del_x^\beta\paren{\scr{L}(v)v(t')}\,dt'\\
  + \int_0^te^{-\epsilon(t-t')\Delta^2}\del_x^\alpha f(\cdot,t')\,dt'.
\end{multline*}
Hence,
\begin{multline*}
\norm{\del_x^\alpha\Gamma v}_2 \leq   C\paren{\norm{\del_x^\alpha v_0}_2
  +\int_0^t\norm{\del_x^{\beta'}e^{-\epsilon
      (t-t)\Delta^2}\del_x^\beta\scr{L}(v)}_2\,dt' +
   \int_0^t\norm{\del_x^\alpha f(t')}_2\,dt'}  \\
\leq  C\paren{\norm{v_0}_{H^s} +
\int_0^t\frac{1}{(t-t')^{1/2}\epsilon^{1/2}}
\norm{\del_x^\beta\scr{L}(v)v(t')}_2\,dt'
+ \int_0^t\norm{f(t')}_{H^s}\,dt'}.
\end{multline*}

In order to proceed further we need to turn our attention to
$\del^\beta_x\scr{L}(v)v$.
\begin{lem}\label{gammacontract}  Let $u,v\in X_{M,T}$ and suppose
  that $\abs{\beta}=s-2$ for $s>n+3$, then there exists a $P\in \N$ so
  that $\norm{\del^\beta_x\scr{L}(u)v}_2\leq
  C\norm{v}_{H^s}\paren{1+\norm{u}_{H^s}+\norm{u}_{H^s}^P}$ with
  $0\leq t \leq T$ and $C=C(M,n,s)$.
\end{lem}
\begin{proof}
We estimate term by term,
\begin{multline*}
\del^\beta_x\scr{L}(u)v=\del_x\paren{a_{jk}(\XuDu)\del_{x_j}\del_{x_k}v}+
\del_x^{\beta}\paren{\vec b_1(\XuDu)\cdot\grad v}+\\
\del_x^\beta\paren{\vec{b}_2(\XuDu)\cdot\grad \bar v} +
\del_x^\beta \paren{c_1(\Xu)v} + \del_x^\beta\paren{c_2(\Xu)\bar v}.
\end{multline*}

We start with $c_1(\Xu)$.  Let $\tilde c_1(\Xu)=c_1(\Xu)-c_1(x,t,0,0)$
so that $\tilde c_1(x,t,0,0)=0.$ Then $\del_x^\beta\paren{c_1(\Xu)v} =
\del_x^\beta \paren{\tilde c_1(\Xu)v}
+ \del_x^\beta\paren{c(x,t,0,0)v}$, and the $H^s$ norm of the second
term is clearly bounded by $C\norm{v}_{H^s}$ where $C$ depends on
$c_1$ and $\beta$.  We have,
\[\norm{\del_x^\beta\paren{\tilde c_1(\Xu)v}}_2 \leq
\sum_{\gamma+\delta=\beta}\norm{\del_x^\gamma\paren{\tilde
  c_1(\Xu)}\del_x^\delta v}_2.\]

If $\abs{\delta} < s-n/2$, then $\norm{\del_x^\delta v}_\infty\leq
C\|\del_x^\delta v\|_{H^{s-\abs{\delta}}}\leq C\norm{v}_{H^s}.$  It
follows that $\norm{\del_x^\gamma \tilde c_1(\Xu)\del_x^\delta
  v}_2\leq C\norm{v}_{H^s}\norm{\del_x^\gamma \tilde c_1(\Xu)}_2$.

As $\tilde c_1(x,t,0,0)=0$ and $\tilde c_1\in C_b^{\tilde N}$ it follows that
$\tilde c_1(\Xu)\in H^s$.  Hence $\norm{\del_x^\gamma\tilde c_1(\Xu)}_2\leq
\norm{c_1(\Xu)}_{H^s}\leq C\paren{\norm{u}_{H^s}+\norm{u}_{H^s}^p}$
for some $p\in \N$.

On the other hand, if $\abs{\delta} \geq s-n/2$, then we may not estimate
$\del_x^\delta v$ in $L^\infty$.  Instead we estimate the other factor
in $L^\infty$. Because $\abs{\gamma}+\abs{\delta}=\abs{\beta}=s-2$, we
have that $\abs{\gamma}\leq n/2-2$. Since $s>n-2$, we
have that $s-\abs{\gamma}>n/2$. Therefore,
\begin{multline*}
\norm{\del_x^\gamma\tilde c_1(\Xu)}_\infty \leq
C\norm{\del_x^\gamma\tilde c_1(\Xu)}_{H^{s-\abs{\gamma}}}\\\leq
C\norm{\tilde c_1(\Xu)}_{H^s}\leq
C\paren{\norm{u}_{H^s}+\norm{u}_{H^s}^p}
\end{multline*}
with $P$ as before.  The estimates for $c_2$ work in exactly the same way.

To estimate $\del_x^{\beta}\paren{\vec b_1(\XuDu)\cdot\grad u}$ note
that our assumptions imply\\ $b_1(x,t,0,0,\vec 0, \vec 0)=0.$
Again we have
\[\norm{\del_x^{\beta}\paren{\vec b_1(\XuDu)\cdot\grad v}}_2 \leq
\sum_{i=1}^n\sum_{\gamma+\delta=\beta}\norm{\del_x^\gamma
  b_{1,i}(\XuDu)\del_x^\delta\del_{x_i}v}_2.\]

In this case, if $\abs{\delta} < s-n/2-1$, then we proceed by estimating
$\del_x^\delta\del_{x_i}v$ in $L^\infty$.  If instead
$\abs{\delta}\geq s-n/2-1$, then we get that $\abs{\gamma}\leq n/2-1$.  We
have that $s>n+2$, so that $s-\abs{\gamma}>n/2+1$.
Hence we may estimate $\del_x^\gamma b_{1,i}(\XuDu)$ in $L^\infty$.
Again the estimates for the terms involving $b_2$ work in the same way
as the terms involving $b_1$.

The estimates for terms involving $a_{jk}$ are essentially identically
to those for $c_i$ and $b_i$ except that we need to require $s>n+3$.

\end{proof}

So we know that
\[\norm{\del_x^\alpha\Gamma v}_2 \leq
C\norm{v_0}_{H^s}+C_{M_0}\frac{2t^{1/2}}{\epsilon^{1/2}}M_0(1+M_0^p)+
\int_0^1\|f(t)\|_{H^s}\,dt\]
 and therefore
\[\norm{\Gamma v}_{H^s} \leq
C\norm{v_0}_{H^s}+C_{M_0}(\frac{2t^{1/2}}{\epsilon^{1/2}}+t)M_0(1+M_0^p)+
\int_0^1\|f(t)\|_{H^s}\,dt.\]
By choosing $T$ so that $C_{M_0}\paren{T^{1/2}/\epsilon^{1/2}+T}M_0(1+M_0^p)<\lambda$ then we get that
$\Gamma$ maps $X_{M_0,T}$ to itself.

Now take $u,v\in X_{M_0,T}$.  We wish to show that $\Gamma$ is a
contraction mapping.  We have that
\begin{multline*}\Gamma u(t) - \Gamma v(t) =
\int_0^te^{-\epsilon(t-t')\Delta^2}\paren{\scr{L}(u)u-\scr{L}(v)v}(t')\,dt'
=\\ \int_0^t e^{-\epsilon(t-t')\Delta^2}
\paren{\paren{\scr{L}(u)-\scr{L}(v)}u+\scr{L}(v)\paren{u-v}}\,dt'.
\end{multline*}
To estimate the terms that arise from $\scr{L}(v)\paren{u-v}$ we may
use Lemma \ref{gammacontract} to conclude that
\[\norm{\int_0^te^{-\epsilon(T-t)\Delta^2}\scr{L}(v)\paren{u-v}}_{H^s}\,dt\leq
C\norm{u-v}_{H^s}\paren{\frac{t^{1/2}}{\epsilon^{1/2}}+t}\paren{1+M_0+M_0^p}.\]
So by choosing $T_\epsilon < T$ this last expression is less then
$1/4\norm{u-v}_{H^s}.$

To estimate terms involving $\paren{\scr{L}(u)-\scr{L}(v)}u$ we
proceed in essentially the same way.  For example, to estimate
$\norm{\paren{c_1(\Xu)-c_1(\Xv)}u}_{H^s}$, rewrite the difference as follows
\begin{multline*}
c_1(x,t,u,\bar u)-c_1(x,t,v,\bar u)+c_1(x,t,v,\bar u) -
c_1(x,t,v,\bar v)=\\
\del_{z_1}c_1(x,t,su-(1-s)v,\bar u)u\paren{u-v}+
\del_{z_2}c_1(x,t,v,r\bar u+(1-r)\bar v)u\paren{\bar u-\bar v}.
\end{multline*}
We can see the above two terms are bounded by $C(M_0+M_0^p)\norm{u-v}_{H^s}$.

The other terms work similarly. We conclude that

\[\norm{\paren{\Gamma u-\Gamma v}(t)}_{H^s}\leq
C\paren{\frac{T^{1/2}}{\epsilon^{1/2}}+T}\paren{1+M_0+M_0^p}\norm{u-v}_{H^s}.\]

Choosing $T_\epsilon<T$ appropriately, we see $\Gamma$ is a
contraction mapping.
Hence there is a unique $v^\epsilon\in X_{T_\epsilon,M_0}$ such that
$v^\epsilon$ solves \eqref{NonLinGenS1} with initial data $v_0$.
\end{proof}

The following lemma is useful in verifying the conditions for our
linear estimates which help us get a uniform time of existence.

\begin{lem}\label{2ndorderassumps}
  Let $v\in X_{T,M_0}$ with $v(0)=u_0$, and suppose that $v$ satisfies
  \eqref{NonLinGenS1} then for $s>N+n/2+4$ the
  coefficients $a_{jk}(x,t,v,\bar v, \grad v, \grad \bar v)$ satisfies
  \linref{LRegularity},
  \linref{LElliptic}, \linref{LFlat}, and \linref{LNoTrap}.  Where the
  constant $C$ that appears in these conditions depends on $M_0$ and
  $C_1$ depends on $u_0$.
\end{lem}
\begin{proof}
  Take $s>N+n/2+4$, then $v$ together with all of it's derivatives to
  order $N+1$ are in  $L^\infty$.  This, together with
  \nlinref{NLRegularity}, allows us to verify \linref{LRegularity}.
  The assumptions 
  \linref{LElliptic} and \linref{LNoTrap} follow immediately from
  \nlinref{NLReal2ndOrder}, \nlinref{NLSymmetric},
  \nlinref{NLElliptic}.   and \nlinref{NLNoTrap}.

  It remains to verify \linref{LFlat} and \linref{LFirstOrder}.
  Clearly, $\abs{I-a_{jk}((x,t,v,\bar v,\grad v, \grad\bar v ))}\leq C/\jap{x}^2$ follows from
  \nlinref{NLFlat} and our $L^\infty$ bounds just as in the cases
  above.

  Let $*$ denote $(x,t,v,\bar v,\grad v, \grad\bar v )$, and consider
  \[\del_{x_i}a_{jk}(*)=\diff{a}{x_i}(*) +
  \diff{a_{jk}}{v}(*)\diff{v}{x_i} + \cdots +
  \diff{a_{jk}}{\del_{x_n}v}(*)\frac{\del^2 \bar v}{\del x_i \del
    x_n}.\]
  The first and second order derivatives of  $v$ are in $L^\infty$ because
  $s> n/2+2$.  Hence by using \nlinref{NLFlat} we can bound each term
  by $C/\jap{x}^2$.

  The estimate for $\del_t a_{jk}(x,t,v,\bar v,\grad v, \grad\bar v)$
  is similar.  The primary difference is that we have to estimate terms
  of the form $\del_tv$ and $\del_t(\del_{x_i} v)$ in $L^\infty$.  To
  handle $\del_tv$ it is enough to notice that $\del_t v$ is equal to
  the right hand side of \eqref{NonLinGenS1}. Each term of
  $\scr{L}(v)v$ is in $L^\infty$ by \nlinref{NLRegularity} and our
  $L^\infty$ bound on $v$ and it's derivatives.  Since $f(x,t)\in
  L^\infty_tH^s_x$  it is in $L^\infty_{t,x}$.  To
  handle the final term $\del_t\del_{x_i} v$ we apply $\del_{x_i}$ to
  our equation and get
\begin{equation*}
\begin{aligned}
&\del_t \del_{x_i}v = -\epsilon\Delta^2 \del_{x_i} v +
ia_{jk}(*)\del_{x_j x_k}(\del_{x_i}v) +
i\diff{a_{jk}}{x_i}(*)\del_{x_j x_k}v +
i\diff{a_{jk}}{v}(*)\del_{x_i}v\\ &+i\diff{a_{jk}}{\bar
  v}(*)\del_{x_i}\bar v
 +i \sum_{l=1}^m \paren{\diff{a_{jk}}{\del_l v}(*)\del_{x_j x_k}v}\del_{x_l x_i}v
+i\sum_{l=1}^m \paren{\diff{a_{jk}}{\del_l \bar v}(*)\del_{x_j x_k}v}\del_{x_l x_i}\bar v\\
&+\vec b_1(*)\cdot\grad \del_{x_i}v + \cdots + \del_{x_i} f(x,t).
\end{aligned}
\end{equation*}

We find that each of these terms may again be handled by our $L^\infty$
bounds for $v$ and its derivatives and \nlinref{NLRegularity}.  Again
$\del_{x_i}f\in L^\infty_t H^{s-1}_x$ so we may bound $\norm{\del_{x_i} f}_\infty.$

Lastly, to bound $\del_t\del_{x_i}a_{jk}(x,t,v,\bar v, \grad v \grad
\bar v)$ we proceed in the same way.  We additionally have to estimate
terms of the form $\del_t\del_{x_i}\del_{x_j}v$ in $L^\infty$, but we
simply apply $\del_{x_j}\del_{x_i}$ to our equation, and again we only
encounter terms involving the derivatives of our coefficients
evaluated at $v$ multiplied by derivatives of $v$ of order less than
4, so for $s>n/2+4$ we may estimate these terms as before.
\end{proof}

To see that our first order terms will satisfy the conditions of our
linear theory we first need two lemma's.

\begin{lem}\label{bintegrable}
  Suppose $b(x,t, z_1,\ldots, z_{2n+2})\in C^{\tilde N}_b(\R\times\R\times
  B^{2n+2}_M(0))$  satisfies \\$b(x,t,0,0,\vec 0,\vec 0)=0$ and
  $\del_{z_i} b(x,t,0,0,\vec 0,\vec 0)=0$.  For any $M\in
  \N$, if $s>n/2+M+1$ and $\tilde N>s+2$  then $b(\XuDu)\in W^{1,M}$
  for $u\in L^\infty_tH^s_x$
\end{lem}
\begin{proof}
    First to see it is in $L^1$, we set $f(r)=b(x,t,ru,r\bar u,r\grad
  u,r\grad \bar u).$  Then we calculate
  \begin{multline*}
  f'(r)=\diff{b}{z_1}(x,t,ru,r\bar u,r\grad u,r\grad \bar u)u +
  \diff{b}{z_2}(x,t,ru,r\bar u,r\grad u,r\grad \bar u)\bar u +\\
  \sum_{i=1}^n\diff{b}{z_{i+2}}(x,t,ru,r\bar u,r\grad u,r\grad \bar u)\diff{u}{x_i}
  + \sum_{i=1}^n\diff{b}{z_{i+n+2}}(x,t,ru,r\bar u,r\grad u,r\grad \bar
  u)\diff{\bar u}{x_i}
  \end{multline*}
  Clearly $f(0)=f'(0)=0,$ and $f(1)=b_1(x,t,u,\bar u,\grad
  u,\grad\bar u)$.
  Now, \[\norm{b}_1=\norm{f(1)}_1=\norm{\int_0^1(1-r)f''(r)\,dr}_1\leq \int_0^1(1-r)\norm{f''(r)}_1\,dr.\]

With in $f''(r)$ are terms of the form $(\del^2 b) u^2$, $(\del^2 b) u\bar
u$, $(\del^2 b) u\del u$, $(\del^2 b) \bar u\del u$,\\ $(\del^2 b) u\del \bar
u$, etc.  The key observation is that they all involve exactly of
degree two when looked at as polynomials in the derivatives of $u$.
So we may apply Cauchy-Schwartz and integrate in $s$.  For example

\begin{multline*}\int \abs{\frac{\del^2 b}{\del z_1\del z_3}(x,t,u,\bar u,\grad u,
  \grad \bar u)u\diff{u}{x_1}}\,dx\leq \norm{\frac{\del^2 b}{\del
    z_1 \del z_3}}_\infty\int\abs{u\diff{u}{x_1}}\,dx\\
\leq \norm{\frac{\del^2 b}{\del
    z_1\del z_3}}_\infty \norm{u}_2\norm{\diff{u}{x_1}}_2\leq
C_{b,u}\norm{u}_{H^s}^2.
\end{multline*}

Estimates of $\del_x^\alpha b$ work similarly.  In fact,
\[\del_{x_i} b=\diff{b}{x_i}(\cdot)+\diff{b}{z_1}(\cdot)\diff{u}{x_i}+
  \diff{b}{z_2}(\cdot)\diff{\bar u}{x_i} +
  \sum_{j=1}^n\diff{b}{z_{j+2}}(\cdot)\frac{\del^2u}{\del x_i\del x_j}
 + \sum_{j=1}^n\diff{b}{z_{j+n+2}}(\cdot)\frac{\del^2 u}{\del x_i\del x_j}\]

Each term has the property that if it is evaluated at $u=0$ it is 0,
as well a derivative in the $z_1$-$z_{2n+2}$.  Hence we may apply the
argument above to bound the $L^1$ norm of each of these.
\end{proof}

The above lemma together with the following observation of Kenig
et.\ al.\ \cite{CKGPLV1998} allow
us to see that $b_1(\XuDu)$ satisfies our linear assumptions.

\begin{lem}\label{bdecay}
  For $M>N+n$ if $b(x,t)\in W^{1,M}$ uniformly in $t$, then one
  can find $\varphi_\mu(x,t)$ so
  that $\supp \varphi_\mu(\cdot,t) \subset Q_\mu^*$,\\
  $\norm{\varphi_\mu(\cdot,t)}_{C^N_b}\leq 1$, and
  \[b(x,t)=\sum_{\mu\in\Z^n}\alpha_\mu(t)\varphi_\mu(x,t)\text{ with
  }\sum_{\mu\in\Z^n}\abs{\alpha_\mu}\leq c\norm{b}_{W^{1,m}}. \]
\end{lem}
\begin{proof}
  By the Sobolev Imbedding theorem if $M>N+n$ then
  $\norm{b(\cdot,t)}_{C^N(Q_\mu^*)}\leq C\norm{b}_{W^{1,M}(Q_\mu^*)}$
  with $C$ independent of $\mu$.  Let $ \eta_\mu$ be a $C^\infty$
  partition of unity subordinate to $Q^*_\mu$ with
  $\norm{\eta_\mu}_{C^N}$ independent of $\mu$.  Then
  $b(x,t)=\sum_{\mu\in\Z^n}\eta_\mu b(x,t)$ and $\norm{\eta_\mu
    b}_{C^N}\leq C \norm{b}_{W^{1,M}(Q_\mu^*)}$.  Since $Q_\mu^*$ have
  bounded overlap $\sum_{\mu\in \Z^n}\norm{b}_{W^{1,M}(Q_\mu^*)}\leq
  C\norm{b}_{W^{1,M}}$.  Hence we just have to
  set \[\varphi_\mu(x,t)=\frac{b(x,t)\eta_\mu(x)}{\norm{\eta_\mu
      b(\cdot,t)}_{C^N}}\text{ and } \alpha_\mu(t) = \norm{\eta_\mu
    b(\cdot,t)}_{C^N}.\]
\end{proof}



Let $J=(1+\Delta)^{\frac{1}{2}}$.  In order to get the necessary
estimates on $\norm{u(t)}_{H^s}$ we inductively estimate
$J^{2m}u$.
Again let $*$ denote $(x,t,u^\epsilon,\bar u^\epsilon,\grad
u^\epsilon,\grad \bar u^\epsilon)$.  As in \cite{CKGPLV2004} we
consider the following systems, for $m=1, 2, \ldots s/2$,
\[
\begin{aligned}
  \del_t J^{2m} u^\epsilon =& -\epsilon \Delta^2 J^{2m}u^\epsilon
  + \scr{L}(u^\epsilon)J^{2m}u^\epsilon +
   2mi\del_{x_l}\paren{a_{jk}(*)}\del^3_{jkl}J^{2\paren{m-1}}u^\epsilon \\
   +& i\del_{jk}^2u^\epsilon\del_{\del_l u}a_{jk}(*) \del_lJ^{2m}u^\epsilon
   + i\del^2_{jk}u^\epsilon\del_{\del_l \bar u}a_{jk}(*)\del_lJ^{2m}\bar u^\epsilon\\
   +&\del_j u^\epsilon \paren{\del_{\del_l u}b_{1,j}(*)\del_lJ^{2m}u^\epsilon
   +\del_{\del_l \bar u}b_{1,j}(*)\del_lJ^{2m}\bar u^\epsilon}\\
   +&\del_j \bar u^\epsilon \paren{\del_{\del_l u}b_{2,j}(*)\del_lJ^{2m}u^\epsilon
   +\del_{\del_l \bar u}b_{2,j}(*)\del_lJ^{2m}\bar u^\epsilon} \\
   +&c_{1,2m}(x,t,\paren{\del^\beta u^\epsilon}_{\abs{\beta}\leq 4},\paren{\del^\beta\bar u^\epsilon}_{\abs{\beta}\leq 4})R_{2m,1}J^{2m}u^\epsilon\\
 +& c_{2,2m}(x,t,\paren{\del^\beta u^\epsilon}_{\abs{\beta}\leq 4},\paren{\del^\beta\bar u^\epsilon}_{\abs{\beta}\leq 4})R_{2m,2}J^{2m}\bar u^\epsilon \\
 +&f(x,t,\paren{\del^\beta u^\epsilon}_{\abs{\beta}\leq
   2m-2},\paren{\del^\beta u^\epsilon}_{\abs{\beta}\leq 2m-2}) +J^{2m}f(x,t)
\end{aligned}
\]

Or more briefly,
\begin{multline*}
\del_t J^{2m}u^\epsilon = -\epsilon\Delta^2
J^{2m}u^\epsilon+\scr{L}_{2m}(u^\epsilon)J^{2m}u^\epsilon
\\+f_{2m}(x,t,\paren{\del^\beta u^\epsilon}_{\abs{\beta}\leq
  2m-2},\paren{\del^\beta u^\epsilon}_{\abs{\beta}\leq 2m-2})
\end{multline*}

Where
\begin{align*}
&\scr{L}_{2m}(u)v=ia_{jk}(x,t,u,\bar u, \grad u,\grad \bar u)\del^2_{jk}v +
b_{1,1_j}(x,t,\paren{\del^\alpha u}_{\abs{\alpha}\leq 2},\paren{\del^\alpha \bar u}_{\abs{\alpha}\leq 1})\del_{x_j}v + \\&
+\tilde b_{{lk,1}_j}(x,t,u,\bar u, \grad u,\grad \bar u)R_{lk}\del_{x_j}v
+ \vec b_{{2m,2}}(x,t,\paren{\del^\alpha u}_{\abs{\alpha}\leq 2},\paren{\del^\alpha \bar u}_{\abs{\alpha}\leq 1})\cdot \grad\bar v\\&
 + c_{1,2m}(x,t,\paren{\del^\beta u}_{\abs{\beta}\leq 4},\paren{\del^\beta\bar u}_{\abs{\beta}\leq 4})R_{2m,1}v\\&
 + c_{2,2m}(x,t,\paren{\del^\beta u}_{\abs{\beta}\leq 4},\paren{\del^\beta\bar u}_{\abs{\beta}\leq 4})R_{2m,2}\bar v,
\end{align*}
with
\begin{multline*}
b_{1,1_j}(x,t,u,\bar u, \grad u,\grad \bar u) =
b_{1_j}(x,t,u,\bar u, \grad u,\grad \bar u) +
i\del_{lk}^2u\del_{\del_j u}a_{lk}(x,t,u,\bar u, \grad u,\grad \bar u)
\\+\del_l u\del_{\del_j u}b_{1_l}(x,t,u,\bar u, \grad u,\grad \bar u) +
\del_l\bar u\del_{\del_j u}b_{2_l}(x,t,u,\bar u, \grad u,\grad \bar u),
\end{multline*}
$$\tilde b_{lk,1_j}=2mi\del_{x_j}(a_{lk}(x,t,u,\bar u, \grad u,\grad \bar u)),$$
$$R_{lk}=\del^2_{lk}J^{-2}, \quad \text{ and }$$
\begin{multline*}
b_{2m,2_j}(x,t,u,\bar u, \grad u,\grad \bar u) =
b_{2,j}(x,t,u,\bar u, \grad u,\grad \bar u) +
i\del_{lk}^2u\del_{\del_j \bar u}a_{lk}(x,t,u,\bar u, \grad u,\grad \bar u)
\\+\del_l u\del_{\del_j \bar u}b_{1_l}(x,t,u,\bar u, \grad u,\grad \bar u) +
\del_l\bar u\del_{\del_j \bar u}b_{2_l}(x,t,u,\bar u, \grad u,\grad \bar u).
\end{multline*}

The same observations from \cite{CKGPLV2004} apply.  The principal
part of $\scr{L}_{2m}(u^\epsilon)$ is independent of $m$.  The
coefficients $b_{{1,1}_j}$, $b_{{2m,2}_j}$, and $\tilde b_{lk}$ depend
on the coefficients $a_{jk}$, $\vec b_l$ and their first derivatives,
$u$ and the derivatives of $u$, but only on $m$ as a multiplicative
constant.  Notice that here both $a_{jk}$ and $b_2$ generate first
order terms but the $\Psi$DO's $R_{lk}$ are independent of $m$.

We need to verify that these coefficients satisfy the conditions for
our linear theory when we evaluate them at any solution $v\in
X_{T,M_0}$ with $v(0)=u_0$.  Since the leading order coefficients have
not changed, Lemma~\ref{2ndorderassumps} still assures us that our
linear assumptions are verified.  Because $s>N+n/2+4$ our $H^s$ bounds
on $v$ together with \nlinref{NLRegularity} give us that the other
coefficients verify \linref{LRegularity}.  Now we just need to verify
\linref{LFirstOrder}.  Notice that in our linear theory we had the
equation in divergence form and hence we have to add an additional
first order term to be able to apply the theory.

\begin{lem}
  The first order coefficients
 $\vec b_{1,1_j}(x,t,v,\bar v,\grad v, \grad\bar v)$, \\
 $\tilde b_{lk,1_j}(x,t,v,\bar v,\grad v, \grad\bar v)R_{lk}$
 and $\del_{x_l}\paren{a_{jk}(x,t,v,\bar v,\grad v,\grad\bar v)}$,
 satisfy \linref{LFirstOrder}.
\end{lem}
\begin{proof}
  Let $M$ be as in Lemma~\ref{bdecay}.  Then by \nlinref{NLFirstOrder}
  we may apply directly apply Lemma~\ref{bintegrable} to get that $b_{1_j}\in
  W^{1,M}$.  Similarly we may apply Lemma~\ref{bintegrable} to $\del_t
  b_{1_j}$.  And hence $b_{1,j}$ has the necessary decomposition for
  \linref{LFirstOrder}.   For the terms involving $\del_{\del_ju}b_{k_l}$ $(k=1,2)$
  notice that if
  $b_k(x,t,0,0,\vec 0, \vec 0)=\del_{z_i}b(x,t,0,0,\vec 0,\vec 0) = 0$
  then $\paren{z_i\del_{z_l}b_k(x,t,\vec z)}|_{z=0} =
  \del_{z_m}\paren{z_i\del_{z_l}b_k(x,t,\vec z)}|_{z=0}=0$. So we may
  again apply Lemma~\ref{bintegrable} and in the same way as for
  $b_{1,j}$ get the desired decomposition for these terms.

  The bounds for
  $\del_{x_l}\paren{a_{jk}(x,t,v,\bar v, \grad v, \grad\bar v)}$ and
  $\del_t\del_{x_l}\paren{a_{jk}(x,t,v,\bar v,\grad v, \grad\bar v)}$ follow from
  \nlinref{NLFirstOrder} together with the $L^\infty$ bounds for
  $\del_{x_l} u$, $\del_{x_l} \bar u$, $\del_{x_lx_i}^2 u$ and
  $\del_{x_lx_i}^2 \bar u$.  Similarly with $b_{lk,1_j}(x,t,v,\bar v,
  \grad v, \grad\bar v)R_{lk}$ and $i\del_{lk}^2v\del_{\del_j
    u}a_{lk}(x,t,v,\bar v, \grad v,\grad \bar v)$.
\end{proof}

For $J^{2m}u^\epsilon$, observe that if we evaluate our coefficients at any $v\in
X_{M_0,T}$ with $v(0)=u_0$ we arrive at a linear equation whose
solution satisfies Theorem~\ref{aplin} with $A_m$ depending on $u_0$ and
the behavior of the coefficients for the system of $J^{2m}u^{\epsilon}$
at $t=0$.  Let $A=\max A_m$ and take
$M_0=20A\lambda$.  Notice at each stage the terms that come from
$f_{2m}$ depend only on terms of order strictly
less then $2(m-1)$, which have been estimated in a previous step in
$L^\infty_TL^2_x$ and so appear with a factor of $T$ in front when we
apply our a priori estimate.

Thus there is a $T'$ independent of $\epsilon$ so that for a fixed
increasing function $R$, so that
\[\sup_{[0,T']}\norm{u^\epsilon(\cdot,t)}_s\leq A\paren{\lambda+T'R(M_0)}\]
We may choose $T'$ small enough so that $A\paren{\lambda+T'R(M_0)}\leq
M_0/4=5A\lambda$.  Then, by our remark after
Theorem~\ref{contractmapping}, we can reapply our contraction mapping
theorem with initial data $u(T_\epsilon)$.  We obtain a solution until
time $2T_\epsilon$, if we apply our linear theory again (on the whole
interval $[0,2T_\epsilon]$ we see that
$\norm{u(2T_\epsilon)}_s\leq M_0/4.$ Then we may continue $k$ times as
long as $kT_\epsilon < T'$.

We thereby extend $u^\epsilon$ to a solution on
$[0,T_0]$ with $u^\epsilon \in X_{T_0,M_0}$ for any $\epsilon$.
Finally we come to the last result.

\begin{thm}
  There exists $u\in C([0,T^*];H^{s-1}(\R^n))\cap
  L^\infty([0,T^*];H^s(\R^n))$ such that
  $u^\epsilon\to u$ as $\epsilon\to 0$ in $C([0,T^*]; H^{s-1})$
\end{thm}

\begin{proof}
  Let $\epsilon, \epsilon'\in (0,1)$ with $\epsilon'<\epsilon$.  Let
  $v=u^\epsilon-u^{\epsilon'}$.  Then $v$ satisfies
  \begin{equation*}
    \left\{\begin{aligned}
      & \del_t v=-(\epsilon-\epsilon')\Delta^2 u^\epsilon - \epsilon'\Delta^2 v
      +\scr{L}(u^{\epsilon'})v
      +\paren{\scr{L}(u^\epsilon)-\scr{L}(u^{\epsilon'})}u^\epsilon \\
      & v(0,x)=0
     \end{aligned}\right.
  \end{equation*}

Now we rewrite
$(\scr{L}(u^{\epsilon})-\scr{L}(u^{\epsilon'}))u^{\epsilon}$ we
proceed term by term

\[
\begin{aligned}
  &ia_{jk}(x,t,u^\epsilon,\bar u^\epsilon,\grad u^\epsilon, \grad\bar u^\epsilon)-
  ia_{jk}(x,t,u^{\epsilon'},\bar u^{\epsilon'},\grad u^{\epsilon'}, \grad\bar u^{\epsilon'})\\
  =&\frac{ia_{jk}(x,t,u^\epsilon,\bar u^\epsilon,\grad u^\epsilon, \grad\bar u^\epsilon) -
    ia_{jk}(x,t,u^{\epsilon'},\bar u^\epsilon,\grad u^\epsilon, \grad\bar u^\epsilon)}{u^\epsilon - u^{\epsilon'}}\del^2_{jk}u^\epsilon v \\
  +&\frac{ia_{jk}(x,t,u^{\epsilon'},\bar u^\epsilon,\grad u^\epsilon, \grad\bar u^\epsilon) -
   ia_{jk}(x,t,u^{\epsilon'},\bar u^{\epsilon'},\grad u^\epsilon, \grad\bar u^\epsilon)}{\bar u^\epsilon - \bar u^{\epsilon'}}\del^2_{jk}u^\epsilon \bar v \\
+&\frac{ia_{jk}(x,t,u^{\epsilon'},\bar u^{\epsilon'},\grad u^\epsilon, \grad\bar u^\epsilon) -
   ia_{jk}(x,t,u^{\epsilon'},\bar u^{\epsilon'},\grad u^{\epsilon'}, \grad\bar u^\epsilon)}{\del_l u^\epsilon - \del_l u^{\epsilon'}}\del^2_{jk}u^\epsilon \del_l v \\
+& \frac{ia_{jk}(x,t,u^{\epsilon'},\bar u^{\epsilon'},\grad u^{\epsilon'}, \grad\bar u^\epsilon) -
   ia_{jk}(x,t,u^{\epsilon'},\bar u^{\epsilon'},\grad u^{\epsilon'}, \grad\bar u^{\epsilon'})}{\del_l \bar u^\epsilon - \del_l \bar u^{\epsilon'}}\del^2_{jk}u^\epsilon \del_l \bar v \\
\end{aligned}
\]

So we get zeroth and first order terms in $v$.
The first order terms have coefficients
$\del_{z_k}a_{jk}(x,t,u^{\epsilon'},\bar u^{\epsilon'}, \cdot, \grad \bar
u^\epsilon)$ and $\del_{z_k}a_{jk}(x,t,u^{\epsilon'},\bar u^{\epsilon'}, \grad \bar
u^{\epsilon'},\cdot)$.  By \nlinref{NLFirstOrder} we assumed the
necessary decomposition of $\del_{z_l} a_{jk}$ and
$\del_t\del_{z_l}a_{jk}$ so that these terms satisfy
\linref{LFirstOrder}.

We apply the same idea to the $b_l$, $l=1,2$ and also get zeroth and
first order terms in $v$.  To see that our first order terms still
satisfy the required estimate we remark that the conclusion of Lemma
4.4 still holds for $\del_{z_k}b_{1,j}(x,t,u^{\epsilon'},\bar
u^{\epsilon'}, \cdot, \grad\bar u^{\epsilon})\del_ju^\epsilon$.  Indeed when we estimate the $L^1$ norm will still have the
product of two elements of $H^s$ whose norm is controlled by $M_0.$
Similarly with the first order terms.

Thus we arrive at a system whose coefficients satisfy the conditions for our
linear estimates.

Applying our linear estimates we conclude that

\[\sup_{[0,T^*]}\norm{v}_2\leq
C(\epsilon-\epsilon')\int_0^{T^*}\norm{\Delta^2 u^\epsilon}_2\,dt\leq C(\epsilon-\epsilon')T_0M_0.\]

Hence as $\epsilon-\epsilon'\to 0$ we have
$u^{\epsilon}-u^{\epsilon'}\to 0$ in $C([0,T^*];L^2).$ So there is a
$u\in C([0,T^*];L^2)$ such that $u^{\epsilon}\to u.$  Since
$u^\epsilon\in L^\infty([0,T^*];H^s)$ and $L^\infty([0,T^*];H^s)$ is the
dual of $L^1([0,T^*]; H^{-s})$ we know there is a subsequence that has
a limit in $L^\infty([0,T^*];H^s).$  But by our first estimate this
could only be $u.$

To see that $u\in C([0,T^*];H^{s-1})$ we simply notice that
\[\norm{u(t)-u(t')}_{H^{s-1}}\leq
\norm{u(t)-u(t')}_2^{1/s}\norm{u(t)-u(t')}_{H^s}^{(s-1)/s}.\]  The
first term in the right hand side tends to 0 and the second is
bounded.  Hence $u\in C([0,T^*];H^{s-1})$.

To see that $u$ is unique we reapply the last argument with
$\epsilon=\epsilon'=0$.  We will end up with
\[\sup_{[0,T^*]}\norm{v}_2=0.\] and therefore $u$ is unique.
\end{proof}



\end{document}